\documentclass[reqno]{amsart}

\usepackage [mathscr]{eucal}
\usepackage{amsmath,amsthm,amssymb,amscd}
\usepackage{amsrefs}
\usepackage{epsf}
\usepackage{amsfonts}
\usepackage{dsfont}
\usepackage{latexsym}
\usepackage{mathrsfs}
\usepackage{layout}
\usepackage{bm}
\usepackage{blkarray}

\newtheorem{theorem}{Theorem}[section]
\newtheorem{lemma}[theorem]{Lemma}
\newtheorem{proposition}[theorem]{Proposition}
\newtheorem{corollary}[theorem]{Corollary}

\newtheorem{definition}[theorem]{Definition}

\numberwithin{equation}{section}

\def\Xint#1{\mathchoice
{\XXint\displaystyle\textstyle{#1}}%
{\XXint\textstyle\scriptstyle{#1}}%
{\XXint\scriptstyle\scriptscriptstyle{#1}}%
{\XXint\scriptscriptstyle\scriptscriptstyle{#1}}%
\!\int}
\def\XXint#1#2#3{{\setbox0=\hbox{$#1{#2#3}{\int}$ }
\vcenter{\hbox{$#2#3$ }}\kern-.6\wd0}}

\def\dint{\Xint-}

\DeclareMathOperator *{\essosc}{ess\ osc}
\DeclareMathOperator *{\osc}{osc}
\DeclareMathOperator *{\esssup}{ess\ sup}
\DeclareMathOperator *{\essinf}{ess\ inf}
\DeclareMathOperator *{\di}{div} 
\DeclareMathOperator *{\meas}{meas}
\DeclareMathOperator *{\dist}{dist}
\DeclareMathOperator *{\data}{data}
\DeclareMathOperator *{\diam}{diam}
\DeclareMathOperator *{\loc}{loc}
\DeclareMathOperator *{\Lip}{Lip}
\DeclareMathOperator *{\Tr}{Tr}
\DeclareMathOperator *{\BMO}{BMO}
\DeclareMathOperator *{\Proj}{Proj}
\DeclareMathOperator *{\BBR}{\mathbb{R}}
\DeclareMathOperator *{\BBC}{\mathbb{C}}

\begin{document}
\title[Dirichlet elliptic systems with a Carleson condition]{The $L^p$ Dirichlet boundary problem for second order Elliptic Systems with rough coefficients}

\author{Martin Dindo\v{s}}
\address{School of Mathematics, \\
         The University of Edinburgh and Maxwell Institute of Mathematical Sciences, UK}
\email{M.Dindos@ed.ac.uk}

\author{Sukjung Hwang}
\address{Department of Mathematics\\ Yonsei University, Republic of Korea}
\email{sukjung\_hwang@yonsei.ac.kr}

\author{Marius Mitrea}
\address{Department of Mathematics \\
         Baylor University, USA}
\email{Marius\_Mitrea@baylor.edu}

\begin{abstract}
Given a domain above a Lipschitz graph, we establish solvability results for strongly elliptic second-order 
systems in divergence-form, allowed to have lower-order (drift) terms, with $L^p$-boundary data for $p$ near $2$ 
(more precisely, in an interval of the form $\big(2-\varepsilon,\frac{2(n-1)}{n-2}+\varepsilon\big)$ for some small $\varepsilon>0$). 
The main novel aspect of our result is that the coefficients of the operator do not have to be constant, or have 
very high regularity, instead they will satisfy a natural Carleson condition that has appeared first in the scalar case. 
A significant example of a system to which our result may be applied is the Lam\'e system for isotropic inhomogeneous materials. We show that our result applies to isotropic materials with Poisson ratio $\nu<0.396$.

Dealing with genuine systems gives rise to substantial new challenges, absent in the scalar case. Among other things,  
there is no maximum principle for general elliptic systems, and the De Giorgi - Nash - Moser theory may also not apply. 
We are, nonetheless, successful in establishing estimates for the square-function and the 
nontangential maximal operator for the solutions of the elliptic system described earlier, 
and use these as alternative tools for proving $L^p$ solvability results for $p$ near $2$.
\end{abstract}

\maketitle

\section{Introduction}\label{S:Intro}

This paper is motivated by the known results concerning boundary value problems 
for second-order elliptic equations in divergence-form, when the coefficients satisfying 
certain natural, minimal smoothness conditions (see \cite{DPP}, \cite{DPR}, \cite{KP}). 

Let $\Omega\subset{\BBR}^{n}$ be the domain lying above the graph of a real-valued Lipschitz function $\phi$ defined 
in ${\mathbb{R}}^{n-1}$, i.e., 
\begin{equation}\label{Eqqq-1}
\Omega=\{(x_0,x'):\,x_0>\phi(x')\}.
\end{equation}
Consider a second-order elliptic system in divergence-form, acting on vector-valued functions 
$u:\Omega\to{\BBR}^N$ according to 
\begin{equation}\label{ES}
\mathcal{L}u=\left[ \partial_{i} \left(A_{ij}^{\alpha \beta}(x) \partial_{j} u_{\beta}\right)
+B_{i}^{\alpha \beta}(x) \partial_{i}u_{\beta}\right]_{\alpha}
\end{equation}
with the summation convention over repeated indices in effect for $i,j\in\{0,\ldots,n-1\}$ and $\alpha,\beta\in\{1,\ldots,N\}$. 
When $N=1$ the operator ${\mathcal{L}}$ is scalar; see \cite{DPP} for a detailed treatment of this case. 

There are many differences between second-order elliptic equations and elliptic systems. 
In general, there is no maximal principle for elliptic systems, and the De Giorgi - Nash - Moser theory 
that provides interior H\"older regularity for scalar elliptic operators may no longer hold.

This causes number of new challenges to be dealt with. For example, it forces us to work with 
a weaker version of the nontangential maximal function, defined using the $L^2$ averages. Also, 
the lack of a maximum principle renders the natural $L^\infty$ end-point for $L^p$-solvability results 
unavailable, thus preventing us from interpolating between solvability at $L^2$ and $L^\infty$ level. 
This means that $L^p$ solvability results for $p\ne 2$ have to be obtained using different methods.

Given a coefficient tensor $A=[A^{\alpha \beta}_{ij}]$ with measurable entries defined in $\Omega$, 
we shall say that $A$ is strongly elliptic if there exist constants $0<\lambda\leq\Lambda<\infty$ such that
\begin{equation}\label{EllipA}
\lambda|\eta|^{2}\leq\sum_{\alpha,\beta=1}^{N}\sum_{i,j=0}^{n-1} 
A_{ij}^{\alpha\beta}(x)\eta_{i}^{\alpha}\eta_{j}^{\beta}\leq\Lambda|\eta|^{2}
\end{equation}
for all $\eta=(\eta_{i}^{\alpha})\in{\BBR}^{nN}$ and a.e. $x\in\Omega$. Note that this forces 
$A$ to be bounded and that we may take $\Lambda=\|A\|_{L^\infty(\Omega)}$. Traditionally, \eqref{EllipA} is 
usually referred to as the {\bf Legendre} condition. It is the strongest form of ellipticity, and   
it is usually relatively easy to verify, since it has a pointwise formulation.  

For some of our results (such as the $L^2$ solvability) it would suffice to assume somewhat weaker 
integral condition which we formulate on ${\mathbb R}^n_+$. Let ${\mathcal H}_0$ be the subspace of
$L^2({\mathbb R}^{n-1};{\mathbb R}^{Nn})$ consisting of $n\times N$ matrices $(f_j^\alpha)_{j,\alpha}$ with 
the property that $(f_j^\alpha)_{j=1,\dots,n-1}$ is curl-free in ${\mathbb R}^{n-1}$ for each $\alpha$. 
Assume that for some $\lambda>0$ and a.e. $x_0>0$
\begin{equation}\label{EllipIC}
\lambda\sum_{i=0}^{n-1}\sum_{\alpha=1}^N\int_{\mathbb R^{n-1}}\left|f_i^\alpha(x')\right|^2dx'
\le \int_{\mathbb R^{n-1}}A_{ij}^{\alpha\beta}(x_0,x')f_i^\alpha(x') f_j^\beta(x')\,dx',\quad\forall f\in\mathcal H_0.
\end{equation} 
(C.f \cite{AM}).
Finally, in the second half of our paper it suffices to assume an even weaker brand of ellipticity, 
namely the {\bf Legendre-Hadamard} condition to the effect that 
\begin{equation}\label{EllipLH}
\lambda |p|^{2}|q|^2\leq\sum_{\alpha,\beta=1}^{N}\sum_{i,j=0}^{n-1} 
A_{ij}^{\alpha\beta}(x)p^\alpha p^\beta q_i q_j
\end{equation}
for all $p=(p^{\alpha})_\alpha\in{\BBR}^{N}$, $q=(q_i)_i\in{\BBR}^{n}$, and a.e. points $x\in\Omega$.

The main result of this paper establishes the solvability of the $L^{2}-$Dirichlet boundary value problem 
for \eqref{ES} assuming the coefficients $A$ and $B$ satisfy a natural Carleson condition which has been 
considered in the scalar case in \cite{DPP}, \cite{DPR}, and elsewhere. We will also impose certain structural 
assumptions on the tensor $A$ that permits recasting \eqref{ES} into a more convenient form.

\vglue1mm
\noindent{\bf Example.} Consider the Lam\'e operator ${\mathcal L}$ for isotropic inhomogeneous materials
in a domain $\Omega$ with Lam\'e coefficients $\lambda(x)$ and $\mu(x)$. Then for $u:\Omega\to{\mathbb R}^n$ 
in vector notation (c.f. \cite{UW}) $\mathcal L$ has the form
\begin{equation}\label{Lame}
\mathcal Lu=\nabla\cdot\left(\lambda(x)(\nabla\cdot u)I+\mu(x)(\nabla u+(\nabla u)^T) \right).
\end{equation}
This fits the template in \eqref{ES} with the lower-order coefficients $B^{\alpha\beta}_i=0$ 
and the coefficients of the second-order term given by (using the Kronecker symbol notation) 
$$
A_{ij}^{\alpha\beta}(x)=\mu(x)\delta_{ij}\delta_{\alpha\beta}+\lambda(x)\delta_{i\alpha}\delta_{j\beta}
+\mu(x)\delta_{i\beta}\delta_{j\alpha}.
$$
Observe that since
\begin{equation}\label{yr4DDF-111x}
\partial_i(A^{\alpha\beta}_{ij}\partial_ju_\beta)=\partial_j(A^{\alpha\beta}_{ij}\partial_iu_\beta)
-\partial_j(A^{\alpha\beta}_{ij})\partial_iu_\beta+\partial_i(A^{\alpha\beta}_{ij})\partial_ju_\beta
\end{equation}
we may rewrite the operator ${\mathcal L}$ as
\begin{equation}\label{ES-lame}
\mathcal{L}u=\left[ \partial_{i} \left(\overline{A}_{ij}^{\alpha \beta}(x) \partial_{j} u_{\beta}\right)
+\overline{B}_{i}^{\alpha \beta}(x) \partial_{i}u_{\beta}\right]_{\alpha},
\end{equation}
where
\begin{equation}\label{eqLM}
\begin{array}{c}
\overline{A}_{ij}^{\alpha\beta}(x)=\mu(x)\delta_{ij}\delta_{\alpha\beta}
+(\lambda(x)+r(x))\delta_{i\alpha}\delta_{j\beta}+(\mu(x)-r(x))\delta_{i\beta}\delta_{j\alpha}
\\[6pt]
\text{and }\,\,
\overline{B}^{\alpha\beta}_i(x)=\partial_jr(x)(\delta_{i\alpha}\delta_{j\beta}-\delta_{i\beta}\delta_{j\alpha}),
\end{array}
\end{equation}
for any $r(x)\in L^\infty$. The introduction of the auxiliary function $r$ infuses an extra degree of flexibility.  

\vglue1mm

The literature on the solvability of boundary value problems for elliptic systems in domains of $\mathbb R^n$ is limited except 
when the tensor $A$ has constant coefficients, or at least smooth enough so that methods like boundary layer potentials may 
be employed. For the solvability the $L^p$-Dirichlet problem for constant coefficients second-order elliptic systems in the 
range $2-\varepsilon<p<2+\varepsilon$ see \cites{DKV, F, FKV, G, BM} and \cite{K}. It was subsequently shown in \cites{S1,S2} 
that in the constant coefficient case this range may be extended to the interval $2-\varepsilon<p<\frac{2(n-1)}{n-3}+\varepsilon$ 
by exploring the solvability of the Regularity problem. See also \cite{MMMM} and in particular \cite{S3} for more 
recent developments. We take advantage of \cite{S3} to extrapolate from solvability for $p=2$ to the range 
$2\le p < \frac{2(n-1)}{n-2}+\varepsilon$ without needing to establish the solvability of the Regularity problem.

Of notable interest is also paper  \cite{DM} where the Stationary Navier-Stokes system in nonsmooth manifolds
was studied. The authors have established results for $L^p$ solvability of the linearized Stokes operator with 
variable coefficients via the method of layer potentials. Because of the method used, at least H\"older continuity 
of the underlying metric tensor had to be assumed.

Another special case is when $A$ is of block-form. For operators $\mathcal{L}=\mbox{div}(A(x)\nabla\cdot)$ 
associated with block matrices $A$, there are numerous results on the $L^p$-solvability of the Dirichlet, Regularity, 
and Neumann problems. This body of results owes to the solution of the Kato problem, where the coefficients of the 
block matrix are also assumed to be independent of the transverse variable. This assumption is usually referred 
in literature as ``$t$-independent" (in our notation it is the $x_0$ variable). See \cite{AHLMT}, \cite{HM}, 
as well as a series of papers by Auscher, Rosen(Axelsson), and McIntosh for second-order elliptic systems (\cites{AA1, AR2, AAM}). 

There are also solvability results in various special cases, assuming that the solutions satisfy 
De Giorgi - Nash - Moser estimates; see \cite{AAAHK} and \cite{HKMPreg} for example (the latter paper is 
also concerned with operators that are $t$-independent). Finally, there are perturbation results in a multitude 
of special cases, such as \cite{AAM} and \cite{AAH}; the first paper shows that solvability in $L^2$ implies solvability 
in $L^p$ for $p$ near $2$, and the second paper has $L^2$-solvability results for small $L^\infty$ perturbations 
of real elliptic operators when the complex matrix is $t$-independent.

Significantly, in the formulation of our solvability result for elliptic systems we shall not assume ``$t$-independence".
Instead, we assume the coefficients $A$ and $B$ satisfy a natural Carleson condition that has appeared 
in the literature so far only for real scalar elliptic operators (\cite {KP01}, \cite{DPP}, and \cite{DPR}). 
The Carleson condition on $A$, formulated in \eqref{Car_hatAA} below, holds uniformly on Lipschitz sub-domains, 
and is therefore a natural condition in the context of chord-arc domains as well. However, in this work we do not 
go beyond the class of Lipschitz domain. Our main result reads as follows.

\begin{theorem}\label{S3:T1} 
Let $\Omega$ be the Lipschitz domain $\{(x_0,x')\in{\mathbb{R}}\times{\mathbb{R}}^{n-1}:\,x_0>\phi(x')\}$.
Denote its Lipschitz constant by $L=\|\nabla\phi\|_{L^\infty}$, fix some $a\in(0,1/L)$, and write $\delta(x)$ 
for the distance from points $x\in{\mathbb{R}}^n$ to $\partial\Omega$. Assume that the coefficient tensor $A$ 
of the operator \eqref{ES} is strongly elliptic with constants $\lambda,\Lambda$ {\rm (}cf. \eqref{EllipA}{\rm )}. 
In addition, assume that:
\begin{itemize}
\item[$(i)$] One has $A_{0j}^{\alpha\beta}=\delta_{\alpha\beta}\delta_{0j}$.
\item[$(ii)$] The following is a Carleson measure in $\Omega$:
\begin{equation}\label{Car_hatAA}
d{\mu}(x)=\left[\left(\sup_{B_{\delta(x)/2}(x)}|\nabla{A}|\right)^{2}
+\left(\sup_{B_{\delta(x)/2}(x)}|{B}|\right)^{2} \right]\delta(x)\,dx.
\end{equation} 
\end{itemize}

Then there exists a small number $K=K(\lambda,\Lambda,n,N)>0$ such that if
\begin{equation}\label{Small-Cond}
\max\big\{L\,,\,\|\mu\|_{\mathcal C}\big\}\leq K
\end{equation}
it follows that $L^p$-Dirichlet problem for the system \eqref{E:D} is solvable for 
whenever $2-\varepsilon< p<\frac{2(n-1)}{n-2}+\varepsilon$ and the estimate
\begin{equation}\label{Main-Est}
\|\tilde{N}_a u\|_{L^{p}(\partial \Omega)}\leq C\|f\|_{L^{p}(\partial \Omega;{\BBR}^N)}
\end{equation}
holds for all energy solutions $\mathcal Lu=0$ with datum $f$. 
Here $\varepsilon=\varepsilon(\lambda,\Lambda,n,N,K,a)>0$ is a small number and $C=C(\lambda,\Lambda,n,N,\Omega,K)>0$ 
is a finite constant independent of $f$. 
\end{theorem}

\noindent{\it Remark.} We will elaborate in Section~2 on the manner in which any operator of the form \eqref{ES} 
may be rewritten so that the condition demanded in $(i)$ holds. In particular, it will follow that Theorem \ref{S3:T1} 
applies to the operator \eqref{Lame}, provided the rewritten system is strongly elliptic.

\vglue1mm
\noindent{\it Remark 2.} It is of considerable interest to replace the condition \eqref{Car_hatAA} 
by another weaker condition, to the effect that the following measure is Carleson in $\Omega$:
\begin{equation}\label{Car_hatAA-osc}
d\widetilde{\mu}(x)=\left[\left(\osc_{B_{\delta(x)/2}(x)}{A}\right)^{2}\delta^{-1}(x)
+\left(\sup_{B_{\delta(x)/2}(x)} |{B}|\right)^{2}\delta(x)\right]\,dx,
\end{equation}
where $\mbox{osc}_B A=\max_{i,j,\alpha,\beta}\left[\sup_B A_{ij}^{\alpha\beta}-\inf_B A_{ij}^{\alpha\beta}\right]$.
In the scalar case this may be done based on the Carleson condition \eqref{Car_hatAA} and Dahlberg-Kenig perturbation 
result for real and scalar elliptic operators. In the case of systems a similar perturbation result is not available yet. 
We address this issue in a subsequent paper \cite{D1}.
\vglue1mm
\noindent{\it Remark 3.} As alluded to earlier, if $\Omega={\mathbb R}^n_{+}$ then 
Theorem \ref{S3:T1} remains valid if in place of the strong ellipticity condition 
\eqref{EllipA} one assumes the integral condition \eqref{EllipIC}. 
When $\Omega$ is an arbitrary Lipschitz domain since \eqref{EllipIC} does not behave well under the pull-back mapping 
discussed in section \ref{SS:PT} we shall require the strong ellipticity assumption.

In particular, we can apply our main theorem to the Lam\'e system. We get the following:

\vglue1mm
\begin{corollary}\label{MM:C1}
Let $\Omega$ be the Lipschitz domain $\{(x_0,x')\in{\mathbb{R}}\times{\mathbb{R}}^{n-1}:\,x_0>\phi(x')\}$ 
with Lipschitz constant $L=\|\nabla\phi\|_{L^\infty}$ and fix some $a\in(0,1/L)$. Assume the Lam\'e coefficients 
$\lambda,\mu\in L^\infty(\Omega)$ satisfy the following two properties:
\begin{itemize} 
\item[$(i)$] There exists $\mu_0>0$ such that 
\begin{equation}\label{Cond-lame}
\mbox{\rm ess }\inf_{x\in\Omega}\{(\sqrt{8}-1)\mu(x)+\lambda(x),(\sqrt{8}+1)\mu(x)-\lambda(x)\}\ge \mu_0.
\end{equation}
\item[$(ii)$] The following is a Carleson measure in $\Omega$: 
\begin{equation}\label{Car_lame}
d{\nu}(x)=\sup_{B_{\delta(x)/2}(x)}\left(|\nabla{\lambda}|+|\nabla{\mu}|\right)^{2}\delta(x)\,dx.
\end{equation}
\end{itemize}

Then exist two small numbers, $K=K(\mu_0,\|\lambda\|_{L^\infty},\|\mu\|_{L^\infty},n)>0$ along with 
$\varepsilon=\varepsilon(\mu_0,\|\lambda\|_{L^\infty},\|\mu\|_{L^\infty},n,K)>0$, such that if 
\begin{equation}\label{Small-Cond2}
\max\big\{L\,,\,\|\nu\|_{\mathcal C}\big\}\leq K
\end{equation}
and $2-\varepsilon <p<\frac{2(n-1)}{n-2}+\varepsilon$ it follows that $L^p$-Dirichlet problem for the Lam\'e system 
\begin{equation}\label{ES-lame2}
\left\{
\begin{array}{l}
\mathcal Lu=\nabla\cdot\left(\lambda(x)(\nabla\cdot u)I+\mu(x)(\nabla u+(\nabla u)^T) \right)=0 
\,\,\text{in }\,\,\Omega,
\\[4pt]
u(x)=f(x)\,\,\text{ for $\sigma$-a.e. }\,x\in\partial\Omega, 
\\[4pt]
\tilde{N}_a(u) \in L^{p}(\partial \Omega), 
\end{array}
\right.
\end{equation}
is solvable, and each energy solution $u:\Omega\to {\mathbb R}^n$ with datum $f$
satisfies the estimate
\begin{equation}\label{Main-Est-LM}
\|\tilde{N}_a u\|_{L^{p}(\partial \Omega)}\leq C\|f\|_{L^{p}(\partial \Omega;{\BBR}^n)}
\end{equation}
where $C(\mu_0,\|\lambda\|_{L^\infty},\|\mu\|_{L^\infty},n,\|\phi\|_{L^\infty},K,a)>0$ is a finite 
constant independent of the function $f$.
\end{corollary}

\noindent {\it Remark.} We note that it was shown in \cite{BM} that the system \eqref{ES-lame2} satisfies the weakest form of ellipticity - the 
Legendre-Hadamard condition  \eqref{EllipLH} if $\mu>0$ and $\lambda+2\mu>0$. Additionally, physical constraints imply that $\mu>-\frac{2}{n}\lambda$ (as $K=\mu+\frac2n\lambda$ called bulk modulus is positive; $K$ is defined as the ratio of the infinitesimal pressure increase to the resulting relative decrease of the volume).
Hence our condition \eqref{Cond-lame} only imposes one additional assumption, namely that 
$$\lambda<(\sqrt{8}+1)\mu\approx 3.828\mu,$$
or alternatively the Poisson ratio $\nu:=\frac{\lambda}{2(\lambda+\mu)}<0.396$. There are many materials where this holds (for example  aluminium, bronze, steel and many other metals, carbon,  polystyrene, PVC, silicate glasses, concrete, etc) \cite{MR}. Examples of few materials where this assumption fails are gold, lead or rubber. For these three materials $\nu$ is near the incompressibility limit ($\nu=\frac12-$) at which  \eqref{ES-lame2} gives {\rm div}$\,u=0$, i.e., the material is incompressible. Intuitively, as both gold and lead are very soft metals, under pressure they behave as liquids, that is a pressure in one direction will cause them to change shape and stretch in remaining directions in order to preserve volume. Rubber is nearly incompressible with $\nu\approx 0.49$.
\vglue2mm

We shall also establish the following large Carleson norm result showing equivalence between the square and
nontangential maximal functions. 

\begin{theorem}\label{S3:T2} 
Retain the notation and background assumptions made in Theorem~\ref{S3:T1}
{\rm (}in particular, the coefficient tensor $A$ of the operator \eqref{ES} is assumed to be strongly elliptic{\rm )}. 
If $\mu$ defined by \eqref{Car_hatAA} is a Carleson measure in $\Omega$ {\rm (}hence, $\|\mu\|_{\mathcal C}$ is finite 
though not necessarily small{\rm )} then for each exponent $p\in(0,\infty)$ any energy solution $u$ of the problem 
$\mathcal{L}u=0$ in $\Omega$ satisfies 
\begin{equation}\label{RED}
\|\tilde{N}_a(u)\|_{L^p(\partial\Omega)}\approx\|{S}_a(u)\|_{L^p(\partial\Omega)},
\end{equation}
where the implied constants only depend on $n,\,N,\,p,\,\lambda,\,\Lambda,\,a$ and $\|\mu\|_{\mathcal C}$.

In fact, the left-pointing inequality in \eqref{RED} holds under a weaker ellipticity assumption. 
Specifically, assume the coefficient tensor $A$ of the system \eqref{ES} satisfies the Legendre-Hadamard 
condition \eqref{EllipLH} with constants $\lambda,\Lambda$, and assume $\mu$ defined in \eqref{Car_hatAA} 
satisfies $\|\mu\|_{\mathcal C}<\infty$. Then for each exponent $p\in(0,\infty)$ any energy solution $u$ 
of the problem $\mathcal{L}u=0$ in $\Omega$ satisfies 
\begin{equation}\label{RED1}
\|\tilde{N}_a(u)\|_{L^p(\partial\Omega)}\lesssim\|{S}_a(u)\|_{L^p(\partial\Omega)},
\end{equation}
where the implied constant again only depends on $n,\,N,\,p,\,\lambda,\,\Lambda\,,a$ and $\|\mu\|_{\mathcal C}$.
Furthermore, the same conclusion also holds for solutions of the Dirichlet problem \eqref{E:D-strip} on domains $\Omega^h$
with constants independent of chosen parameter $h>0$.
\end{theorem}

\begin{proof}  This follows from Corollary~\ref{S5:C4} and Proposition~\ref{S3:C7} . We shall make appropriate comments 
in the proofs of Corollary~\ref{S5:C4} and Proposition~\ref{S3:C7} where a modified argument is required 
when dealing with the domains $\Omega^h$.
\end{proof}

In this vein we wish to note that the papers \cites{AA1,AR2} have established \eqref{RED1} 
for coefficient tensors $A$ which are $t$-independent.

The paper is organised as follows. In Section~\ref{S2} we introduce basic notions and definitions needed throughout. 
In Section~3 we discus the $L^2$-Dirichlet problem and also give the proof of our main result. In Section~4 we 
establish important estimates for the square-function. Subsequently, in Section~5, we produce similar estimates 
for the nontangential maximal operator. Finally, Section~6 deals with the $L^p$-solvability for $p$ near $2$ 
using extrapolation arguments. In section 7 we then discuss Corollary \ref{MM:C1}.

\section{Definitions and background results}
\label{S2}

For a vector-valued function $u=(u_{\alpha})_{1\leq\alpha\leq N}:\Omega\to{\BBR}^{N}$ we let 
$\nabla u$ denote the Jacobian matrix of $u$. The latter is defined as the matrix with entries
\begin{equation}\label{Eqqq-2} 
\left(\nabla u\right)_{i}^{\alpha}=\partial_{i} u_{\alpha} 
=\frac{\partial u_{\alpha}}{\partial x_{i}}
\end{equation}
for $i\in\{0,\ldots,n-1\}$ and $\alpha\in\{1,\ldots,N\}$.

Given an open set $\Omega\subseteq{\mathbb{R}}^n$, for $0\leq k\leq\infty$ we use $C^{k}(\Omega;{\BBR}^{N})$ 
to denote the space of all ${\BBR}^{N}$-valued functions in $\Omega$ with continuous partial derivatives up to 
order $k$. Also, we shall let $C^{k}_{0}(\Omega;{\BBR}^{N})$ be the subspace of $C^{k}(\Omega;{\BBR}^{N})$ 
consisting functions that are compactly supported in $\Omega$. For $k\in{\mathbb{N}}$ and $1\leq p<\infty$, 
let $W^{k,p}(\Omega;{\BBR}^{N})$ be the Sobolev space which is the collection of ${\BBR}^{N}$-valued locally 
integrable functions in $\Omega$ having distributional derivatives of order $\leq k$ in $L^p(\Omega;{\BBR}^{N})$. 
When $k=1$, equip this space with the norm
\begin{equation}\label{EFFV} 
\|u\|_{W^{1,p}(\Omega)}:=\left[\int_{\Omega}\left(|u(x)|^{p}+|(\nabla u)(x)|^{p}\right)\,dx\right]^{1/p}.
\end{equation}
Also, let $W^{k,p}_{\rm loc}(\Omega;{\BBR}^{N})$ stands for the local version of $W^{k,p}(\Omega;{\BBR}^{N})$. 
Similarly, we denote by $\dot{W}^{k,p}(\Omega;{\BBR}^{N})$ the homogeneous version of the $L^p$-based Sobolev 
space of order one in $\Omega$. When $k=1$, this is endowed with the semi-norm
\begin{equation}\label{EFFV2} 
\|u\|_{\dot{W}^{1,p}(\Omega)}:=\left[\int_\Omega|(\nabla u)(x)|^{p}\,dx\right]^{1/p}.
\end{equation}

Throughout this paper, by a weak solution of \eqref{ES} in $\Omega$ we shall understand a function 
$u\in W^{1,2}_{\rm loc}(\Omega;{\BBR}^{N})$ satisfying $\mathcal{L}u=0$ in the sense of distributions 
in $\Omega$. 

\subsection{Non-tangential maximal and square functions}
\label{SS:NTS}

Consider a domain of the form 
\begin{equation}\label{Omega-111}
\Omega=\{(x_0,x')\in\BBR\times{\BBR}^{n-1}:\, x_0>\phi(x')\},
\end{equation}
where $\phi:\BBR^{n-1}\to\BBR$ is a Lipschitz function with Lipschitz constant given by 
$L:=\|\nabla\phi\|_{L^\infty(\BBR^{n-1})}$. For each point $x\in{\mathbb{R}}^n$ abbreviate 
$\delta(x):=\mbox{dist}(x,\partial\Omega)$. In particular, %
\begin{equation}\label{PTFCC}
\delta(x)\approx x_0-\phi(x')\,\,\text{ uniformly for }\,\,x=(x_0,x')\in\Omega.
\end{equation}

A cone {\rm (}or non-tangential approach region{\rm )} of aperture $a\in(0,\infty)$ 
with vertex at the point $Q=(x_0,x')\in{\mathbb{R}}\times{\mathbb{R}}^{n-1}$ is defined as
\begin{equation}\label{TFC-6}
\Gamma_{a}(Q)=\big\{y=(y_0,y')\in{\mathbb{R}}\times{\mathbb{R}}^{n-1}:\,a(y_0-x_0)>|x'-y'|\big\}.
\end{equation}
Imposing the demand that $a\in(0,1/L)$ then ensures that $\Gamma_{a}(Q)\subseteq\Omega$ whenever $Q\in\partial\Omega$.
In particular, when $\Omega=\BBR^n_+$ all parameters $a\in(0,\infty)$ may be considered.
Sometimes it is necessary to truncate $\Gamma_{a}(Q)$ at height $h$, in which scenario we write
\begin{equation}\label{TRe3}
\Gamma_{a}^{h}(Q):=\Gamma_{a}(Q)\cap\{x\in\Omega:\,\delta(x)\leq h\}.
\end{equation}

\begin{definition}\label{D:S}
For $\Omega \subset \mathbb{R}^{n}$ as above and $a\in(0,1/L)$, the square function of some 
$u\in W^{1,2}_{\rm loc}(\Omega; {\BBR}^{N})$ is defined at each $Q\in\partial\Omega$ by
\begin{equation}\label{yrdd}
S_{a}(u)(Q):=\left(\int_{\Gamma_{a}(Q)}|(\nabla u)(x)|^{2}\delta(x)^{2-n}\,dx\right)^{1/2}
\end{equation}
and, for each $h>0$, its truncated version is given by 
\begin{equation}\label{yrdd.2}
S_{a}^{h}(u)(Q):=\left(\int_{\Gamma_{a}^{h}(Q)}|(\nabla u)(x)|^{2}\delta(x)^{2-n}\,dx\right)^{1/2}.
\end{equation}
\end{definition}

A simple application of Fubini's theorem gives 
\begin{equation}\label{SSS-1}
\|S_{a}(u)\|^{2}_{L^{2}(\partial\Omega)}\approx\int_{\Omega}|(\nabla u)(x)|^{2}\delta(x)\,dx.
\end{equation}

\begin{definition}\label{D:NT.a} 
For $\Omega\subset\mathbb{R}^{n}$ as above and $a\in(0,1/L)$, the nontangential maximal function of some 
$u\in C^{\,0}(\Omega;{\BBR}^{N})$ and its truncated version at height $h$ are defined at each $Q\in\partial\Omega$ by
\begin{equation}\label{SSS-2}
N_{a}(u)(Q):=\sup_{x\in\Gamma_{a}(Q)}|u(x)|\,\,\text{ and }\,\,
N^h_{a}(u)(Q):=\sup_{x\in\Gamma^h_{a}(Q)}|u(x)|.
\end{equation}
\end{definition}

Moreover, we shall also consider a related version  of the above nontangential maximal function.
This is denoted by $\tilde{N}_a$ and is defined using $L^2$ averages over balls in the domain $\Omega$. 
Specifically, we make the following definition.

\begin{definition}\label{D:NT.b} 
For $\Omega\subset\mathbb{R}^{n}$ as above and $a\in(0,1/L)$,
given $u\in L^2_{\rm loc}(\Omega;{\BBR}^{N})$ we set
\begin{equation}\label{SSS-3}
\tilde{N}_{a}(u)(Q):=\sup_{x\in\Gamma_{a}(Q)}w(x)\,\,\text{ and }\,\,
\tilde{N}_{a}^{h}(u)(Q):=\sup_{x\in\Gamma_{a}^{h}(Q)}w(x)
\end{equation}
for each $Q\in\partial\Omega$ and $h>0$ where, at each $x\in\Omega$, 
\begin{equation}\label{w}
w(x):=\left(\dint_{B_{\delta(x)/2}(x)}|u|^{2}(z)\,dz\right)^{1/2}.
\end{equation}
\end{definition}

Here and elsewhere, a barred integral indicates integral average.
We note that, given $u\in L^2_{\rm loc}(\Omega;{\BBR}^{N})$, the function $w$ 
associated with $u$ as in \eqref{w} is continuous and $\tilde{N}_a(u)=N_a(w)$ 
everywhere on $\partial\Omega$. For systems with bounded measurable coefficients, the best regularity we can expect 
from a weak solution of \eqref{ES} is $u\in W^{1,2}_{\rm loc}(\Omega;\BBR^N)$. 
In particular, $u$ might not be pointwise well-defined. In the scalar case $N=1$ 
by the De Giorgi-Nash-Moser estimates the situation is different as the solutions are locally H\"older continuous. 
Hence, while in the scalar case considering $N_a$ typically suffices, in the case of systems the consideration of 
$\tilde{N}_a$ becomes necessary. Note that our condition \eqref{Car_hatAA} implies that $u$ exhibits better regularity, 
as $A$ has a locally bounded gradient. From this once then deduces that $u\in W^{2,2+\varepsilon}_{\rm loc}(\Omega;\BBR^N)$ for 
some $\varepsilon>0$, hence our $u$ does have well-defined pointwise values in dimensions $n=2,3,4$. In a subsequent paper 
we will consider operators satisfying the weaker condition \eqref{Car_hatAA-osc}; for such operators weak solutions $u$ 
have, in general, well-defined pointwise values only when $n=2$.

\subsection{The Carleson measure condition}
\label{SS:Car} 

We begin by recalling the definition of a Carleson measure in a domain $\Omega$ as in \eqref{Omega-111}. 
For $P\in{\BBR}^n$, define the ball centered at $P$ with the radius $r>0$ as
\begin{equation}\label{Ball-1}
B_{r}(P):=\{x\in{\BBR}^n:\,|x-P|<r\}.
\end{equation}
Next, given an arbitrary location $Q \in \partial\Omega$ along with a scale $r>0$, we shall abbreviate 
$\Delta=\Delta_{r}=\Delta_{r}(Q)=\partial\Omega\cap B_{r}(Q)$ and refer to this as the 
surface ball centered at $Q$ and of radius $r$. The Carleson region $T(\Delta_r(Q))$ associated with such a surface ball 
is then defined by
\begin{equation}\label{tent-1}
T(\Delta_{r}(Q)):=\Omega\cap B_{r}(Q).
\end{equation}

\begin{definition}\label{Carleson}
Let $\Omega$ be as in \eqref{Omega-111}.
A Borel measure $\mu$ in $\Omega$ is said to be Carleson if it has finite Carleson norm, i.e., 
\begin{equation}\label{CMC-1}
\|\mu\|_{\mathcal C}=\sup_{\Delta}\frac{\mu\left(T(\Delta)\right)}{\sigma(\Delta)}<\infty,
\end{equation}
where the supremum is taken over all surface balls $\Delta\subseteq\partial\Omega$, and 
where $\sigma$ is the surface measure on $\partial\Omega$. 
\end{definition}

As regards the elliptic operator introduced in \eqref{ES}, in all that follows we 
shall assume that the coefficients $A$ and $B$ satisfy the following natural conditions.  
First, we assume that the entries $A_{ij}^{\alpha \beta}$ of $A$ are in ${\rm Lip}_{\rm loc}(\Omega)$ 
(locally Lipschitz) while the entries $B_{i}^{\alpha \beta}$ of $B$ are in $L^\infty_{\rm loc}(\Omega)$. 
Second, we assume that
\begin{equation}\label{CarA}
d\mu(x)=\left[\left(\sup_{B_{\delta(x)/2}(x)}|\nabla A|\right)^{2}
+\left(\sup_{B_{\delta(x)/2}(x)}|B|\right)^{2}\right]\delta(x) \,dx
\end{equation}
is a Carleson measure in $\Omega$. Occasionally (but not everywhere) we will additionally 
assume that its Carleson norm $\|\mu\|_{\mathcal{C}}$ is sufficiently small. 
The following result plays a significant role in the future.

\begin{proposition}\label{T:Car}
Let $\Omega$ be as in \eqref{Omega-111} and fix some $a\in(0,1/L)$. Given a function $f\in L^\infty_{\rm loc}(\Omega)$, define 
$d\nu=f\,dx$ and $d\mu(x)=\left[\esssup_{B_{\delta(x)/2}(x) }|f|\right]dx$. Assume that $\mu$ is a Carleson measure in $\Omega$. 
Then there exists a finite constant $C=C(L,a)>0$ such that for every $u\in L^{2}_{\rm loc}(\Omega;{\BBC})$ one has
\begin{equation}\label{Ca-222}
\int_{\Omega}|u(x)|^2\,d\nu(x)\leq C\|\mu\|_{\mathcal{C}} 
\int_{\partial\Omega}\left(\tilde{N}_{a}(u)\right)^2\,d\sigma.
\end{equation}
\end{proposition}

\begin{proof} Let 
$$
\Omega=\bigcup_{i}{\mathcal O}_i
$$
be a Whitney decomposition of $\Omega$. Without loss of generality, 
assume that the Whitney sets ${\mathcal O}_i$ are such that ${\mathcal O}_i\subset B_{\delta(x)/2}(x)$
for each $x\in {\mathcal O}_i$. Also, recall that $|{\mathcal O}_i|\approx |B_{\delta(x)/2}(x)|$ where, generally speaking,  
$|E|$ denotes the Lebesgue measure of the Lebesgue measurable set $E\subseteq{\mathbb{R}}^n$. 
It follows that on each ${\mathcal O}_i$ we have
$$
\int_{{\mathcal O}_i}|u(x)|^2\,d\nu(x)\leq \left[\esssup_{{\mathcal O}_i}|f|\right] \int_{{\mathcal O}_i}|u(x)|^2\,dx.
$$
By the definition \eqref{w} for $w$ it follows that for any $y\in {{\mathcal O}_i}$ we have
$$
\int_{{\mathcal O}_i}|u(x)|^2\,d\nu(x)\lesssim\left[\esssup_{B_{\delta(y)/2}(y) }|f|\right] w(y)^2|{\mathcal O}_i|.
$$
Integrating in $y\in {{\mathcal O}_i}$ we then conclude that 
$$
\int_{{\mathcal O}_i}|u(x)|^2\,d\nu(x)\lesssim \int_{{\mathcal O}_i} w(y)^2d\mu(y).
$$
Summing over all $i$ we obtain
$$
\int_{\Omega}|u(x)|^2\,d\nu(x)\lesssim \int_{\Omega} w(y)^2d\mu(y)
\lesssim \|\mu\|_{\mathcal C}\int_{\partial\Omega}N_a(w)^2d\sigma,
$$
where the last inequality follows from the usual inequality for Carleson measures. 
Since $\tilde{N}_{a}(u)=N_a(w)$ the claim follows.
\end{proof}

The aforementioned assumptions on coefficients of the system \eqref{ES} 
are compatible with the change of variables described in the next two subsections. 

\subsection{Reformulations of \eqref{ES} and ellipticity}
\label{SS:Nor}

In this section, we rewrite the elliptic system \eqref{ES} in a more convenient form. 
As we allow one derivative to fall on the tensor $A$ we shall assume that this has entries in 
$\text{Lip}_{\rm loc}(\Omega)$.

Let $A_{ij}=[A_{ij}^{\alpha \beta}]_{\alpha,\beta}$ for $i,j\in\{0,1,\ldots, n-1\}$. 
Hence, each $A_{ij}$ is an $N\times N$ matrix. It is natural to assume that $A_{00}$ 
(the principle minor of $A$) is invertible. For example, this is guaranteed whenever 
\eqref{EllipA} or \eqref{EllipLH} holds. To proceed, consider the tensors 
$\widehat{A}=[\widehat{A}_{ij}^{\alpha\beta}]_{i,j,\alpha,\beta}$ and 
$\widehat{B}=[\widehat{B}_{i}^{\alpha\beta}]_{i,\alpha,\beta}$ defined by
\begin{equation}\label{hatA}
\widehat{A}^{\alpha\beta}_{ij}:=\sum_{\gamma=1}^{N}\left[A_{00}^{-1}\right]^{\alpha\gamma}A_{ij}^{\gamma\beta},
\end{equation}
and
\begin{equation}\label{hatB}
\widehat{B}^{\alpha\beta}_{i}:=\sum_{\gamma=1}^{N}\left(\left[A_{00}^{-1}\right]^{\alpha \gamma} 
B_{i}^{\gamma\beta}-\sum_{k=0}^{n-1}\partial_k\left([A_{00}^{-1}]^{\alpha\gamma}\right)A_{ki}^{\gamma\beta}\right).
\end{equation}

Then the original regularity assumptions made on $A,B$ are retained for \eqref{hatA}-\eqref{hatB}, i.e., 
$\widehat{A}$ and $\widehat{B}$ have entries in $\text{Lip}_{\rm loc}(\Omega)$ and $L^\infty_{\rm loc}(\Omega)$, respectively.
In addition, observe that $\widehat{A}$ is diagonalized in the $x_0$ variable, namely $\widehat{A}_{00}=I_{N\times N}$. Let
\begin{equation}\label{hatL}
\widehat{\mathcal{L}}u:=\left[\partial_{i}\left(\widehat{A}_{ij}^{\alpha\beta}(x)\partial_{j}u_{\beta}\right)
+\widehat{B}_{i}^{\alpha\beta}(x)\partial_{i}u_{\beta}\right]_{\alpha}.
\end{equation}
If $u\in W^{1,2}_{\rm loc}(\Omega;{\BBR}^N)$ satisfies $\mathcal{L}u=0$ in $\Omega$ then  
\begin{align}\label{hatL1}
\widehat{\mathcal{L}}u &=\sum_{\gamma}[A_{00}^{-1}]^{\alpha\gamma} 
\left[\partial_{i}\left(A_{ij}^{\gamma\beta}\partial_{j}u_{\beta}\right) 
+B_{i}^{\gamma\beta}\partial_{i}u_{\beta}\right] 
\nonumber\\[4pt]
&\quad +\sum_{\gamma}\partial_{i}\left([A_{00}^{-1}]^{\alpha\gamma}\right)A_{ij}^{\gamma\beta}\partial_{j}u_{\beta}
-\sum_{\gamma}\partial_k\left([A_{00}^{-1}]^{\alpha\gamma}\right)A_{ki}^{\gamma\beta}\partial_iu_\beta
\nonumber\\[4pt]
&= 0.
\end{align}
Here we have used the equation $\mathcal L u=0$ for the first two terms of \eqref{hatL1}, and that 
the two terms in the second line cancel (as may be seen by permuting indices $i\to j\to k\to i$ 
in the last term). 

For technical reasons we will also require that $\widehat{A}_{0j}=0_{N\times N}$ for $j>0$. 
This may be achieved as follows. Since
\begin{equation}\label{yr4DDF-111}
\partial_0(\widehat A^{\alpha\beta}_{0j}\partial_ju_\beta)=\partial_j(\widehat A^{\alpha\beta}_{0j}\partial_0u_\beta)
-\partial_j(\widehat A^{\alpha\beta}_{0j})\partial_0u_\beta+\partial_0(\widehat A^{\alpha\beta}_{0j})\partial_ju_\beta,
\end{equation}
it follows that if we define a new coefficient tensor by taking
\begin{equation}\label{yr4DDF}
\overline{A_{ij}^{\alpha\beta}}:=
\begin{cases}
\widehat A_{ij}^{\alpha\beta},&\quad\mbox{if $i,j>0$ or $i=j=0$,}
\\[6pt]
\widehat A_{ij}^{\alpha\beta}+\widehat A_{ji}^{\alpha\beta},&\quad\mbox{if $i>0$ and $j=0$,}
\\[6pt]
0,&\quad\mbox{if $i=0$ and $j>0$,}
\end{cases}
\end{equation}
then the difference $\widehat{\mathcal L} u-\partial_i\left(\overline{A_{ij}^{\alpha\beta}}\partial_ju_\beta\right)$
consists of just first-order terms. This is due to the fact that 
$$
\widehat A_{ij}^{\alpha\beta}+\widehat A_{ji}^{\alpha\beta}=\overline{A_{ij}^{\alpha\beta}}+\overline{A_{ji}^{\alpha\beta}}
$$ 
which, in light of \eqref{yr4DDF-111}, implies that the two operators $\widehat{\mathcal L} u$ and 
$\partial_i\left(\overline{A_{ij}^{\alpha\beta}}\partial_ju_\beta\right)$ have identical second-order terms.
As such, $u$ may be regarded as a solution of a system similar to \eqref{hatL1} with $\overline{A_{ij}^{\alpha\beta}}$ 
replaced by $\widehat{A}_{ij}^{\alpha\beta}$ and appropriately modified coefficients matrix of the first order terms. 

Observe that the coefficients $\overline{A}$, $\overline{B}$ of the new system will still satisfy 
a Carleson condition with norm controlled by $\|\mu\|_{\mathcal{C}}$. That is, if the original 
tensors $A$ and $B$ are such that \eqref{CarA} is a Carleson measure then
\begin{equation}\label{Car_hatA}
d\overline{\mu}(x)=\left[\left(\sup_{B_{\delta(x)/2}(x)}|\nabla\overline{A}(x)|\right)^{2}
+\left(\sup_{B_{\delta(x)/2}(x)}|\overline{B}(x)|\right)^{2}\right]\delta(x)\,dx
\end{equation}
is also a Carleson measure in $\Omega$ whose Carleson norm $\|\overline{\mu}\|_{\mathcal{C}}$ 
may be estimated in terms of the original norm $\|{\mu}\|_{\mathcal{C}}$.

\vglue1mm

Let us now discuss briefly how such coefficient changes affect the strong ellipticity. In general, if \eqref{EllipA} 
holds for $A$, it might not hold anymore for \eqref{hatA}, or \eqref{yr4DDF}. For this reason we will assume 
strong ellipticity for \eqref{yr4DDF}. However, in some situations the strong ellipticity for $A$, $\widehat A$, 
and $\overline{A}$ are equivalent. This is always true when $N=1$, i.e., if the the operator \eqref{ES} is scalar. 

A similar observation can be made for equivalence of strong ellipticities of $\widehat A$ and $\overline A$, when $\widehat A$ 
enjoys the following symmetry property:
$$
\widehat A_{ij}^{\alpha\beta}=\widehat A_{ij}^{\beta\alpha}.
$$
In the case of Lam\'e system this happens when we choose $r=(\mu-\lambda)/2$ in \eqref{eqLM}. 

\subsection{Pullback Transformation}
\label{SS:PT}

For a domain $\Omega$ as in \eqref{Omega-111}, consider the mapping 
$\rho:\mathbb{R}^{n}_{+}\to\Omega$ appearing in works of Ne\v{c}as, Dahlberg, 
Kenig-Stein and others, defined by
\begin{equation}\label{E:rho}
\rho(x_0, x'):=\big(x_0+P_{\gamma x_0}\ast\phi(x'),x'\big),
\qquad\forall\,(x_0,x')\in\mathbb{R}^{n}_{+},
\end{equation}
for some fixed $\gamma>0$. The precise choice of $\gamma$ is not important as longs as the map 
$\rho:\mathbb{R}^{n}_{+}\to\Omega$ is a bijection which is always the case for small choice of $\gamma$, say 
$0<\gamma<\gamma_0\|\nabla\phi\|_{L^\infty(\BBR^{n-1})}$ for some small dimensional constant $\gamma_0>0$.
Also, $P$ appearing in \eqref{E:rho} is a non-negative function $P\in C_{0}^{\infty}(\mathbb{R}^{n-1})$ 
and, for each $\lambda>0$, we agree to abbreviate 
\begin{equation}\label{PPP-1a}
P_{\lambda}(x'):=\lambda^{-n+1}P(x'/\lambda),\qquad\forall\,x'\in{\mathbb{R}}^{n-1}.
\end{equation}
Finally, $P_{\lambda}\ast\phi(x')$ is the convolution
\begin{equation}\label{PPP-lambda}
P_{\lambda}\ast\phi(x'):=\int_{\mathbb{R}^{n-1}}P_{\lambda}(x'-y')\phi(y')\,dy'. 
\end{equation}
Observe that $\rho$ extends up to the boundary of ${\BBR}^{n}_{+}$ and maps one-to-one from 
$\partial {\BBR}^{n}_{+}$ onto $\partial\Omega$. Also, for sufficiently small $\gamma\lesssim L$ 
the map $\rho$ is a bijection from $\overline{\mathbb{R}^{n}_{+} }$ onto $\overline\Omega$ 
and, hence, invertible. 

For $u\in W^{1,2}_{\rm loc}(\Omega;\BBR^N)$ that solves $\mathcal{L}u=0$ in $\Omega$ with Dirichlet 
datum $f$ consider $v:=u\circ\rho$ and $\widetilde{f}:=f\circ\rho$. The change of variables 
via the map $\rho$ just described implies that $v\in W^{1,2}_{\rm loc}(\mathbb{R}^{n}_{+};{\BBR}^{N})$ 
is a null-solution of a system, namely 
\begin{equation}\label{ESvvv}
0=\di\left(\widetilde{A}^\alpha(x)\nabla v\right)+\widetilde{B}^\alpha(x)\cdot\nabla v,
\qquad\mbox{for }\alpha\in\{1,2,\dots,N\},
\end{equation}
with boundary datum $\widetilde{f}$ on $\partial\mathbb{R}^{n}_{+}$. Hence, solving a boundary value 
problem for $u$ in $\Omega$ is equivalent to solving a related boundary value problem for $v$ in 
$\mathbb{R}^{n}_{+}$. Crucially, if the coefficients of the original system are such that \eqref{CarA} 
is a Carleson measure, then the coefficients of $\widetilde{A}$ and $\widetilde{B}$ satisfy an analogous 
Carleson condition in the upper-half space. If, in addition, the Carleson norm of \eqref{CarA} 
is small and $\|\nabla\phi\|_{L^\infty}$ is also small, then the Carleson norm of the measure 
associated in the same manner as in \eqref{CarA} with the new coefficients $\widetilde{A}$ and $\widetilde{B}$ 
will be correspondingly small. It is also not hard to see that strong ellipticity is preserved under this change of variables.

Next, we shall discuss the condition $(i)$ of Theorem \ref{S3:T1} in relation to the pull-back transformation $\rho$. 
Assume that the original tensor satisfies $A_{0j}^{\alpha\beta}=\delta_{0j}\delta_{\alpha\beta}$. After applying the 
pull-back $\rho$, the new tensor $\widetilde{A}$ as in \eqref{ESvvv} for the system on $\mathbb{R}^{n}_{+}$ no longer 
satisfies this property. We can remedy this issue by performing the change of coefficients we have discussed in Subsection~\ref{SS:Nor}.

If the function $\phi$ in \eqref{E:rho} has a small Lipschitz norm it follows that the Jacobian of the map 
$\rho$ is very close to the identity (as a small $L^\infty$ perturbation of $I$, with the size of the perturbation 
depending on the $\|\nabla\phi\|_{L^\infty}$). Because of this, the coefficients $\widetilde{A}$ after the pull-back 
will have $|\widetilde{A}_{0j}^{\alpha\beta}-\delta_{0j}\delta_{\alpha\beta}|$ small.
This smallness is preserved after performing the change of coefficients from Subsection~\ref{SS:Nor}.
Hence, if the original tensor $A$ satisfies the strong ellipticity condition, so will do the final tensor of the 
corresponding system in $\mathbb{R}^{n}_{+}$, as long as $\|\nabla\phi\|_{L^\infty}$ is sufficiently small.

The bottom line is that taking the pull-back under the map $\rho$ allows us to reduce the task of solving 
the Dirichlet problem for \eqref{ES} in $\Omega$ to the special case when the underlying domain is 
$\Omega={\BBR}^n_+$ and both conditions $(i)$ and $(ii)$ of Theorem \ref{S3:T1} hold in ${\BBR}^n_+$.

\subsection{Basic estimates}
\label{SS:Ineq}

Here we recall some of the basic the inequalities valid for weak solutions of the operator 
${\mathcal{L}}$. 

\begin{proposition}\label{poincare}{\rm (}Poincar\'e inequality{\rm )}
There exists a finite dimensional constant $C=C(n)>0$ such that, for all balls $B_{R}\subset{\BBR}^{n}$ 
and all $u\in W^{1,2}(B_{R};{\BBR}^{N})$, 
\begin{gather*}
\int_{B_{R}}|u-u_{B_{R}}|^{2}\,dx\leq CR^{2}\int_{B_{R}}|\nabla u|^{2}\,dx, 
\end{gather*}
where
\begin{equation}\label{uuu-AVE}
u_{B_{R}}:=\dint_{B_{R}}u(x)\,dx.
\end{equation}
\end{proposition}

\begin{proposition}\label{caccio}{\rm (}Interior Cacciopoli inequality{\rm )} 
Let $\Omega$ be as in \eqref{Eqqq-1}, and let $\mathcal L$ as in \eqref{ES} satisfy 
the Legendre-Hadamard condition \eqref{EllipLH}. In addition, assume that there exists some $M\in(0,\infty)$ 
with the property that for a.e. $x\in\Omega$ one has
\begin{equation}\label{AltC}
|\nabla A(x)|\le M\delta^{-1}(x),\qquad |B(x)|\le M\delta^{-1}(x).
\end{equation}

Then there exists a finite positive constant $C=C(n,N,\lambda,\Lambda, M)>0$ such that if $B_{4R}\subset\Omega$
and $u\in W^{1,2}(B_{2R};{\BBR}^{N})$ solves $\mathcal Lu=0$ in $B_{2R}$ it follows that 
\begin{gather*}
\int_{B_{R}}|\nabla u|^{2}\,dx\leq CR^{-2}\int_{B_{2R}}|u|^{2}\,dx. 
\end{gather*}
\end{proposition}

\begin{proof} 
Consider a smooth cutoff function $\varphi$ such that $\varphi=1$ on $B_R$ and
$\varphi$ vanishes outside $B_{2R}$. Let $v=u\varphi$. The original claim readily follows
once we establish
\begin{equation}\label{eq342}
\int_{B_{2R}}|\nabla v|^{2}\,dx\leq CR^{-2}\int_{B_{2R}}|u|^{2}\,dx.
\end{equation}
Note that $v\in W^{1,2}_0(B_{2R})$. Suppose first that the coefficients of $\mathcal L$ are constant on the ball $B_{2R}$. 
As in \cite[p.\,9]{Y}, via the via the Fourier transform techniques, it follows that 
\begin{equation}\label{ccoeff}
\int_{B_{2R}}|\nabla v|^{2}\,dx\approx 
\int_{B_{2R}}A_{ij}^{\alpha\beta}\partial_jv_\beta\partial_iv_\alpha\,dx,
\end{equation}
where the implicit constants depend only on the constant in the Legendre-Hadamard condition.
We now adapt the idea in \cite[pp.\,11-12]{Y} to the variable coefficient case. Fix $\rho\in(0,1)$ and cover the ball 
$B_{2R}$ by a finite collection of smaller balls $B_{R\rho}(x_k)$ with centers $x_1,x_2,\dots,x_N=N(\rho)$ inside $B_{2R}$, 
where the number of balls $N(\rho)$ in this cover depends only on the dimension $n$ and chosen parameter $\rho$. 
As in \cite{Y} one can construct a family of functions $(\varphi_k)_{1\leq k\leq N}$ such that each $\varphi_k$ 
is supported on $B_{2R\rho}(x_k)$, $|\nabla \varphi_k|\le 1/R$, and $\sum_k\varphi_k^2=1$. 
Consider the functions $v_k=v\varphi_k$. By the constant coefficient result we have that 
\begin{equation}\label{ccoeff33}
\int_{B_{2R}}|\nabla v_k|^{2}\,dx\approx \int_{B_{2R\rho}(x_k)}
A_{ij}^{\alpha\beta}(x_k)\partial_j(v_k)_\beta\partial_i(v_k)_\alpha\,dx.
\end{equation}
On each ball $B_{2R\rho}(x_k)$ the condition imposed in \eqref{AltC} implies small oscillation of the coefficients of $A$, 
that is, $\osc_{B_{2R\rho }} A\le 2M\rho$. In turn, this permits us to estimate 
\begin{align}\label{ccoeff34}
\int_{B_{2R}}|\nabla v_k|^{2}\,dx &\lesssim \int_{B_{2R\rho}(x_k)}A_{ij}^{\alpha\beta}
\partial_j(v_k)_\beta\partial_i(v_k)_\alpha\,dx
\\[6pt]
&\quad+\int_{B_{2R\rho}(x_k)}(A_{ij}^{\alpha\beta}(x_k)-A_{ij}^{\alpha\beta})\partial_j(v_k)_\beta\partial_i(v_k)_\alpha\,dx
\nonumber\\[6pt]
&\le\int_{B_{2R\rho}(x_k)}A_{ij}^{\alpha\beta}\partial_j(v_k)_\beta\partial_i(v_k)_\alpha\,dx
+2M\rho\int_{B_{2R}}|\nabla v_k|^{2}\,dx.
\nonumber
\end{align}
Hence, choosing $\rho\in(0,1)$ sufficiently small ensures that the last term above may be absorbed into the left-hand side. 
For such a choice $\rho\in(0,1)$ we may therefore conclude that 
\begin{equation}\label{ccoeff35}
\int_{B_{2R}}|\nabla v_k|^{2}\,dx\lesssim \int_{B_{2R}}A_{ij}^{\alpha\beta}\partial_j(v_k)_\beta\partial_i(v_k)_\alpha\,dx,
\end{equation}
where the implied constant only depends on the the constant in the Legendre-Hadamard condition.
Summing up in $k$ then yields 
\begin{align}\label{eq17}
\int_{B_{2R}}|\nabla v|^2dx &=\sum_k \int_{B_{2R}}|\nabla v|^2\varphi_k^2\,dx
\nonumber\\[6pt]
&\le 2\sum_k\int_{B_{2R}}\big[|\nabla v_k|^2+|v|^2|\nabla \varphi_k|^2\big]\,dx.
\end{align}
Observe that the second term above may be estimated by $2N(\rho)R^{-2}\int_{B_{2R}}|v|^2dx$, and this  
suits our purposes, given the goal formulated in \eqref{eq342}. To deal with the first term in \eqref{eq17}, 
the idea is to use \eqref{ccoeff35} and the inequality between the arithmetic and geometric means:
\begin{align}\label{TEXAS.babe!}
\sum_k\int_{B_{2R}}|\nabla v_k|^2dx &=\sum_k\int_{B_{2R}}A_{ij}^{\alpha\beta}\partial_j(v_k)_\beta\partial_i(v_k)_\alpha\,dx
\\[6pt]
&\le2\sum_k \int_{B_{2R}}\big[A_{ij}^{\alpha\beta}\varphi_k^2\partial_jv_\beta\partial_iv_\alpha
+|A_{ij}^{\alpha\beta}||v|^2|\nabla \varphi_k|^2\big]\,dx.
\nonumber
\end{align}
The second term above may be controlled by $CR^{-2}\int_{B_{2R}}|v|^2dx$ which, once again, suits our purposes, 
while the first term is just $\int_{B_{2R}}A_{ij}^{\alpha\beta}\partial_jv_\beta\partial_iv_\alpha\,dx$. It follows that
\begin{eqnarray}\label{eqXX33}
\int_{B_{2R}}|\nabla v|^{2}\,dx\lesssim \int_{B_{2R}}A_{ij}^{\alpha\beta}\partial_jv_\beta\partial_iv_\alpha\,dx
+N\Lambda R^{-2}\int_{B_{2R}}|v|^2dx.
\end{eqnarray}
We shall next use the PDE to handle the first term in the right-hand side. Since $v$ vanishes on the boundary 
of $B_{2R}$ after integration by parts using that $v=u\varphi$ we obtain:
\begin{eqnarray}\label{Baylor}
\int_{B_{2R}}A_{ij}^{\alpha\beta}\partial_jv_\beta\partial_iv_\alpha\,dx 
&=&-\int_{B_{2R}}\varphi(\mathcal Lu)_\alpha v_\alpha\,dx+\int_{B_{2R}}\varphi B_i^{\alpha\beta}\partial_iu_\beta v_\alpha\,dx
\nonumber\\
&+&\int_{B_{2R}}A_{ij}^{\alpha\beta}(\partial_j\varphi)u_\beta \partial_iv_\alpha\,dx.
\end{eqnarray}
Since $\mathcal Lu=0$ on $B_{2R}$, the first term in the right-hand side vanishes.
For the second term, given that we have $|B_i^{\alpha\beta}|\le MR^{-1}$, we obtain
\begin{align}\label{eqXX35}
\int_{B_{2R}}\varphi B_i^{\alpha\beta}\partial_iu_\beta v_\alpha\,dx
&=\int_{B_{2R}}B_i^{\alpha\beta}\partial_iv_\beta v_\alpha\,dx
-\int_{B_{2R}}\partial_i\varphi B_i^{\alpha\beta}u_\beta v_\alpha\,dx
\nonumber\\[6pt]
&\lesssim MR^{-1}\|\nabla v\|_{L^2(B_2R)}\|v\|_{L^2(B_{2R})}
\nonumber\\[6pt]
&\quad+MR^{-2}\|u\|_{L^2(B_{2R})}\|v\|_{L^2(B_{2R})}.
\end{align}
Finally, we may estimate the last term in the right-hand side of \eqref{Baylor} as follows:
\begin{eqnarray}\label{eqXX34}
\int_{B_{2R}}A_{ij}^{\alpha\beta}(\partial_j\varphi)u_\beta \partial_iv_\alpha\,dx
\lesssim \Lambda R^{-1}\|\nabla v\|_{L^2(B_2R)}\|u\|_{L^2(B_{2R})}.
\end{eqnarray}
After combining \eqref{eqXX33}-\eqref{eqXX34}, using the arithmetic mean-geometric mean inequality
to handle the first terms in the right-hand side of \eqref{eqXX35} and \eqref{eqXX34}, and also bearing 
in mind that $|v|\le|u|$, we arrive at the conclusion that
\begin{eqnarray}\label{Waco.TX}
\int_{B_{2R}}|\nabla v|^{2}\,dx\le \frac12\int_{B_{2R}}|\nabla v|^{2}\,dx +C(n,N,\lambda,\Lambda,M)R^{-2}\int_{B_{2R}}|u|^{2}\,dx.
\end{eqnarray}
From this, the inequality claimed in \eqref{eq342} readily follows.
\end{proof}

Imposing a slightly stronger ellipticity assumption on $\mathcal L$ allows us to establish 
the following boundary Cacciopoli inequality.

\begin{proposition}\label{caccioB}{\rm (}Boundary Cacciopoli inequality{\rm )} 
Let $\Omega$ be as in \eqref{Eqqq-1} and suppose the system $\mathcal L$ is as in \eqref{ES}. In addition,  
assume the coefficients of $\mathcal L$ satisfy $|B(x)|\le M\delta^{-1}(x)$ for some $M\in(0,\infty)$ 
and a.e. $x\in\Omega$, and that either
\begin{itemize}
\item[(i)] $\mathcal L$ satisfies the Legendre condition \eqref{EllipA}, or
\item[(ii)] $\Omega={\mathbb R}^n_{+}$ and $\mathcal L$ satisfies the condition \eqref{EllipIC}.
\end{itemize}

Then there exists a finite positive constant $C=C(n,N,\lambda,\Lambda,M)>0$ such that if $R>0$ and 
$u\in W^{1,2}(T(\Delta_{2R}))$ satisfies $\mathcal Lu=0$ in $T(\Delta_{2R})$ as well as 
${\rm Tr}\,u=0$ on $\Delta_{2R}$ then
\begin{gather*}
\int_{T(\Delta_{R})}|\nabla u|^{2}\,dx\leq CR^{-2}\int_{T(\Delta_{2R})}|u|^{2}\,dx. 
\end{gather*}
\end{proposition}

This boundary Cacciopoli inequality is needed for the extrapolation argument. 
See \cite[Remark~1.3]{S3} for a proof based on the Hardy's inequality.
See also \cite[Proposition~1.14]{FMZ} for an argument working on domains with more general boundaries (uniform). 
In particular, the proof given in \cite{FMZ} does not require strong ellipticity assumption (i) and works also 
with the integral condition (ii).

\section{The $L^2$-Dirichlet problem}
\label{S3}

In anticipation to formulating the $L^p$-Dirichlet problem we first recall the notion of classical solvability, 
via the Lax-Milgram lemma. Given a domain $\Omega$ as in \eqref{Omega-111}, consider the bilinear form 
$\mathcal B:{\dot{W}}^{1,2}(\Omega;{\BBR}^N)\times {\dot{W}}^{1,2}_0(\Omega;{\BBR}^N) \to\mathbb R$ defined by
\begin{equation}\label{eq-BF}
\mathcal B[u,w]=\int_{\Omega}\left[A_{ij}^{\alpha\beta}\partial_ju_\beta\partial_iw_\alpha
+B_i^{\alpha\beta}\partial_iu_\beta w_\alpha\right]\,dx.
\end{equation}
Clearly, $\mathcal B$ is bounded under the assumptions $A$ has entries in $L^\infty(\Omega)$ and that $B$ satisfies
$|B(x)|\le M\delta^{-1} (x)$. Indeed, for the second term this allows us to use the Cauchy-Schwarz inequality 
followed by an application of Hardy-Sobolev inequality
\begin{equation}\label{HSob}
\int_{\Omega}\frac{|w(x)|^2}{\delta(x)^2}dx\le C\int_\Omega |\nabla w|^2\,dx
\end{equation}
which holds for each function $w\in \dot{W}^{1,2}_0(\Omega;{\BBR}^N)$.
To proceed, let $\dot{B}^{2,2}_{1/2}(\partial\Omega;{\BBR}^N)$ denote the homogeneous Besov space of traces of functions 
in $\dot{W}^{1,2}(\Omega;{\BBR}^N)$. Given an arbitrary $f\in\dot{B}^{2,2}_{1/2}(\partial\Omega;{\BBR}^N)$,  
there exists $v\in \dot{W}^{1,2}(\Omega;{\BBR}^N)$ such that ${\rm Tr}\,v=f$ on $\partial\Omega$. 
Writing $u=u_0+v$, we seek $u_0\in \dot{W}^{1,2}_0(\Omega;{\BBR}^N)$ such that
$$
\mathcal B[u_0,w]=-B[v,w]\,\,\,\mbox{ for all }\,\,w\in \dot{W}^{1,2}_0(\Omega;{\BBR}^N).
$$
Observe that $-B[v,\cdot]\in\big(\dot{W}^{1,2}_0(\Omega;{\BBR}^N)\big)^*$, hence by the Lax-Milgram lemma 
there exists unique solution $u_0\in \dot{W}^{1,2}_0(\Omega;{\BBR}^N)$, provided the form $\mathcal B$ 
is coercive on the space $\dot{W}^{1,2}_0(\Omega;{\BBR}^N)$.

Under the assumption that $\Omega$ is Lipschitz (cf. \eqref{Eqqq-1}) and $\mathcal L$ satisfies the 
Legendre condition \eqref{EllipA} or, alternatively, assuming $\Omega={\mathbb R}^n_+$ and 
that the integral condition  \eqref{EllipIC} holds, we clearly have
$$
\int_{\Omega}A_{ij}^{\alpha\beta}\partial_ju_\beta\partial_iu_\alpha\,dx\ge \lambda \int_\Omega|\nabla u|^2\,dx,
$$
for all $u\in \dot{W}^{1,2}_0(\Omega;{\BBR}^N)$. On the other hand, for the term involving the entries of 
$B$ we may use \eqref{HSob} to estimate 
$$
\left|\int_{\Omega}B_i^{\alpha\beta}\partial_iu_\beta u_\alpha\,dx \right|\le CM\int_\Omega|\nabla u|^2\,dx,
$$
hence
$$
{\mathcal B}[u,u]\ge (\lambda-CM)\|\nabla u\|^2_{L^2(\Omega)},\quad\mbox{for all }u\in \dot{W}^{1,2}_0(\Omega;{\BBR}^N).
$$
This implies coercivity of the bilinear form $\mathcal B$, for small values of $M$. 
Since we are interested in solvability of the Dirichlet problem under the assumption of that 
the Carleson measure of \eqref{Car_hatAA} is small, this is the case since $M\lesssim \|\mu\|^{1/2}_{\mathcal C}$.

It follows that given any $f\in \dot{B}^{2,2}_{1/2}(\partial\Omega;{\BBR}^N)$ there exists a unique 
$u\in \dot{W}^{1,2}(\Omega;{\BBR}^N)$ such that $\mathcal{L}u=0$ in $\Omega$ for $\mathcal L$ 
given by \eqref{ES} and ${\rm Tr}\,u=f$ on $\partial\Omega$. We call such $u$ the {\it energy solution} 
of the elliptic system $\mathcal{L}$ in $\Omega$. With this in hand, we can now define the notion of $L^p$-solvability. 

\begin{definition}\label{D:Dirichlet} 
Let $\Omega$ be the Lipschitz domain introduced in \eqref{Omega-111} and fix an integrability exponent 
$p\in(1,\infty)$. Also, fix a background parameter $a>0$. Consider the following Dirichlet problem 
for a vector valued function $u:\Omega\to{\BBR}^N$:
\begin{equation}\label{E:D}
\left\{
\begin{array}{l}
0=\partial_{i}\left(A_{ij}^{\alpha \beta}(x)\partial_{j}u_{\beta}\right) 
+B_{i}^{\alpha\beta}(x)\partial_{i}u_{\beta}\,\text{ in }\,\Omega,\,\,\,\alpha\in\{1,2,\dots,N\},
\\[6pt]
u(x)=f(x)\,\,\,\text{ for $\sigma$-a.e. }\,\,x\in\partial\Omega, 
\\[6pt]
\tilde{N}_a(u) \in L^{p}(\partial \Omega), 
\end{array}
\right.
\end{equation}
where the usual summation convention over repeated indices {\rm (}$i,j$ and $\beta$ in this case{\rm )} 
is employed. We say the Dirichlet problem \eqref{E:D} is solvable for a given $p\in(1,\infty)$ if there exists a finite constant 
$C=C(\lambda, \Lambda, n, p,\Omega)>0$ such that the unique energy solution $u\in \dot{W}^{1,2}(\Omega;{\BBR}^N)$, 
provided by the Lax-Milgram lemma, corresponding to a boundary datum 
$f\in L^p(\partial\Omega;{\BBR}^N)\cap \dot{B}^{2,2}_{1/2}(\partial\Omega;{\BBR}^N)$, satisfies the estimate
\begin{equation}\label{y7tGV}
\|\tilde{N}_a u\|_{L^{p}(\partial\Omega)}\leq C\|f\|_{L^{p}(\partial\Omega;{\BBR}^N)}.
\end{equation}
In \eqref{E:D} the solution $u$ agrees with $f$ at the boundary in the sense of trace on 
$\dot{W}^{1,2}(\Omega;{\BBR}^N)$ as well as in the sense of a.e. limit \eqref{eq-ae}, as explained below.
\end{definition}

\noindent{\it Remark.} By Lax-Milgram lemma the solution $u$ of \eqref{E:D} is unique 
in the space $\dot{W}^{1,2}(\Omega;{\BBR}^N)$ modulo constants (in ${\BBR}^N$). 
Our additional assumption that at $\sigma$-a.e. point on $\partial\Omega$ we have $u=f\in L^p(\partial\Omega;{\BBR}^N)$  
eliminates the constant solutions and, hence, guarantees genuine uniqueness. Since the space 
$\dot{B}^{2,2}_{1/2}(\partial\Omega;{\BBR}^N)\cap L^p(\partial\Omega;{\BBR}^N)$ is dense in 
$L^p(\partial\Omega;{\BBR}^N)$ for each $p\in(1,\infty)$, it follows that there exists a 
unique continuous extension of the solution operator
\begin{equation}\label{Sol-OP}
f\mapsto u
\end{equation}
to the whole space $L^p(\partial\Omega;{\BBR}^N)$, with $u$ such that $\tilde{N}_a u\in L^p(\partial\Omega)$ 
and the accompanying estimate $\|\tilde{N}_a u \|_{L^{p}(\partial \Omega)} 
\leq C\|f\|_{L^{p}(\partial\Omega;{\BBR}^N)}$ being valid. It is a legitimate question to consider in what sense
we have a convergence of $u$ given by the solution operator \eqref{Sol-OP} to its boundary datum 
$f\in L^{p}(\partial\Omega;{\BBR}^N)$. The answer can be found in the appendix of paper \cite{DP} 
(the proof is given for scalar operators but adapts in a straightforward way to our situation). 
Consider the average $u_{av}:\Omega\to {\mathbb R}^N$ defined by
$$
{u}_{av}(x)=\dint_{B_{\delta(x)/2}(x)} u(y)\,dy,\quad \forall x\in \Omega.
$$
Then 
\begin{equation}\label{eq-ae}
f(Q)=\lim_{x\to Q,\,x\in\Gamma(Q)} u_{av}(x),\qquad\text{for $\sigma$-a.e. }Q\in\partial\Omega.
\end{equation}

We are now ready to establish the main result Theorem \ref{S3:T1}. The solutions to the Dirichlet problem 
in the infinite domain $\Omega=\BBR^n_{+}$ will be obtained as a limit of solutions in
infinite strips $\Omega^h=\{x=(x_0,x')\in\BBR\times{\BBR}^{n-1}:\,0<x_0<h\}$. We define them as follows. 
Keep the same assumptions on the coefficients of $\mathcal L$ as at the beginning of this section 
in order to guarantee unique solvability in the space of energy solutions. Recall that we require 
$|B(x)|\le M\delta^{-1}(x)$ for some small $M$ where $\delta(x)$ measures distance of $x$ to the boundary 
of $\partial\Omega$. It follows that this condition holds also with respect to domains $\Omega^h$ and 
therefore we can also conclude existence and uniqueness of energy solutions on the domains $\Omega^h$ which is needed below.

\begin{definition}\label{D:DirichletStrip} 
Let $\Omega={\BBR}^n_{+}$, and let $\Omega^h$ be the infinite strip 
$$
\Omega^h=\{x=(x_0, x')\in\BBR\times{\BBR}^{n-1}:\,0<x_0<h\},
$$ 
Also, fix some exponent $p\in(1,\infty)$ and pick an aperture parameter $a>0$. 
Let $u$ be a vector-valued function $u:\Omega\to{\BBR}^N$ such that
\begin{equation}\label{E:D-strip}
\left\{
\begin{array}{l}
0=\partial_{i}\left(A_{ij}^{\alpha \beta}(x)\partial_{j}u_{\beta}\right) 
+B_{i}^{\alpha\beta}(x)\partial_{i}u_{\beta}\,\,\text{ in }\,\,\Omega^h,\quad\alpha\in\{1,2,\dots,N\},
\\[6pt]
u(x_0,x')=0\,\,\text{for all }\,\,x_0 \geq h,
\\[6pt]
u(x)=f(x)\,\,\text{ for $\sigma$-a.e. }\,\,x\in\partial\Omega, 
\\[4pt]
\tilde{N}_{a}(u) \in L^{p}(\partial \Omega),
\end{array}
\right.
\end{equation}
with the usual summation convention over repeated indices in effect.

We say the Dirichlet problem \eqref{E:D-strip} is solvable for a given exponent $p\in(1,\infty)$ if 
there exists some finite constant $C=C(p,\Omega)>0$ such that for all boundary data
$f\in L^p(\partial\Omega;{\BBR}^N)\cap \dot{B}^{2,2}_{1/2}(\partial\Omega;{\BBR}^N)$ 
we have that $u\big|_{\Omega^h}$ is the unique ``energy solution" of the problem 
\begin{equation}\label{E:D-energy}
\left\{
\begin{array}{l}
0=\partial_{i}\left(A_{ij}^{\alpha \beta}(x)\partial_{j}u_{\beta}\right) 
+B_{i}^{\alpha\beta}(x)\partial_{i}u_{\beta}\,\,\text{ in }\,\,\Omega^h,\quad\alpha\in\{1,2,\dots,N\},
\\[6pt]
u(x_0,x')=0\,\,\,\text{for }\,\,\,x_0=h,
\\[6pt]
u(x)=f(x)\,\,\text{ for $\sigma$-a.e. }\,\,x\in\partial\Omega, 
\end{array}
\right.
\end{equation}
and satisfies the estimate
\begin{equation}\label{y7tGV-2}
\|\tilde{N}_{a}(u)\|_{L^{p}(\partial\Omega)}\leq C\|f\|_{L^{p}(\partial\Omega;{\BBC})}.
\end{equation}
\end{definition}

\begin{proof}[Proof of Theorem~\ref{S3:T1}] 
As indicated in the previous section there is no loss of generality in assuming that $\Omega=\BBR^n_{+}$, 
that the matrix $A_{00}$ is equal to $I_{N\times N}$, and that $A_{0j}^{\alpha\beta}=0$ for $j>0$ 
(via the pull-back transformation involving the function $\rho$ from Section~\ref{SS:PT}, and 
the change of variables \eqref{hatA} and \eqref{hatB}). The new system will be strongly elliptic 
if the original system was so, thanks to the smallness of Lipschitz constant of the function $\phi$.

\vskip1mm

We will establish the solvability of the Dirichlet problem \eqref{E:D-strip}, applying the results of Sections~4 and 5. 
The constants will not depend on the width of the strip. Then, a limiting argument (sending the width of the domain to infinity) 
proves Theorem~\ref{S3:T1}.

\vskip1mm

Let  $u_h$ be the energy solution in $\Omega^h$ as in Definition \ref{D:DirichletStrip}.
Corollary~\ref{S4:C2} ensures that there exists some finite $C=C(n,N,\lambda,\Lambda)>0$ such that 
\begin{align}\label{S3:L4:E00bbbb}    
\lambda\iint_{\BBR^n_{+}}|\nabla u_h|^{2}x_0\,dx'\,dx_0 
&\le\int_{\BBR^{n-1}}|f(x')|^{2}\,dx'
\nonumber\\[6pt]
&\quad+C\|\mu\|_{\mathcal{C}}\int_{\BBR^{n-1}}\left[\tilde{N}_a(u_h)\right]^{2}\,dx'.
\end{align}
It is important to note that, thanks to our working assumptions on $\Omega_h$ and the fact that $u_h$ is 
an energy solution, we have
$$
\|S_{a}(u_h)\|^2_{L^2({\BBR}^{n-1})}\lesssim\int_{\Omega^h}|\nabla u_h|^2dx<\infty.
$$
While the implicit constant in this estimate does depend on $h$, the estimate itself guarantees 
that the $L^2$ norm of the square function of $u_h$ is finite. Hence, by Theorem~\ref{S3:T2} we have
\begin{align}\label{Lakota}
\lambda\iint_{\BBR^n_+}|\nabla u_h|^{2}x_0\,dx'\,dx_0 
&=C_a\int_{\BBR^{n-1}}\left[S_a(u_h)\right]^{2}\,dx'
\nonumber\\[4pt]
&\approx\int_{\BBR^{n-1}} \left[\tilde{N}_a(u_h)\right]^{2}\,dx'.
\end{align}
From this it follows that
\begin{align}\label{S3:L4:E00bbbbb}    
\int_{\BBR^{n-1}}\left[\tilde{N}_a(u_h)\right]^{2}\,dx'
&\leq C_0\int_{\BBR^{n-1}}|f(x')|^{2}\,dx'
\nonumber\\[6pt]
&\quad+C\|\mu\|_{\mathcal{C}}\int_{\BBR^{n-1}}\left[\tilde{N}_a(u_h)\right]^{2}\,dx'.
\end{align}
The constants in this estimate are independent of $h$. 
Choose $K$ in Theorem~\ref{S3:T1} such that $C\|\mu\|_{\mathcal{C}}<1/2$. Such a choice then entails 
\begin{equation}\label{S3:L4:E00bbbbbbb}    
\int_{\BBR^{n-1}}\left[\tilde{N}_a(u_h)\right]^{2}\,dx'\leq 2C_0\int_{\BBR^{n-1}}|f(x')|^{2}\,dx',
\end{equation}
for all energy solutions $u_h$ of the system \eqref{E:D-strip}.  

We now consider the limit of $u_h$, as $h \to \infty$. For each $h\ge 1$, 
the Lax-Milgram lemma gives 
\begin{equation}\label{WACO}
\|\nabla u_h\|_{L^2({\BBR}^n_+)}\le C\|f\|_{\dot B_{1/2}^{2,2}},
\end{equation}
where the constant $C\in(0,\infty)$ depends on the coercivity constant of the bilinear form $\mathcal B$ 
from \eqref{eq-BF} in $\Omega^h$ (which is uniform in $h$), and $\|A\|_{L^\infty}=\Lambda$. Therefore, this 
bound is uniform in $h$. Keeping this in mid and recalling that $\text{Tr}(u_h)=f$, weak convergence argument 
yields a sub-sequence convergent to some $u$ with $\|\nabla u\|_{L^2({\BBR}^n_{+})}\leq C\|f\|_{\dot B_{1/2}^{2,2}}$ 
and $\text{Tr}(u)=f$. This sub-sequence is therefore strongly convergent to $u$ in $L^2_{\rm loc}({\BBR}^n_+)$ 
by standard functional analysis (if $X,Y$ are Banach spaces and $X$ embeds compactly into $Y$, 
then if a sequence converges weakly in $X$  it must converge strongly in $Y$).

It follows that the $L^2$ averages $w_h$ of $u_h$ converge locally and uniformly to $w$, 
the $L^2$ averages of $u$ in $C_{\rm loc}({\BBR}^n_{+})$. Let $\Gamma_k(x')$ be the doubly truncated 
cone $\Gamma(x')\cap\{1/k<x_0<k\}$. Define 
$$
\tilde{N}_k(u)(x')=\sup_{y\in\Gamma_k(x')}|w(y)|,
$$ 
and consider $\tilde{N}_k(u_h)(x')$ defined analogously. Then we have
$$
\tilde{N}_k(u_h)(x')\to\tilde{N}_k(u)(x')\,\,\,\text{ uniformly on compact subsets $K\subset{\BBR}^{n-1}$.}
$$
Finally, using \eqref{S3:L4:E00bbbbbbb}, this give on each such set $K$,
$$
\|\tilde{N}_k(u)\|_{L^p(K)}=\lim_{h\to\infty}\|\tilde{N}_k(u_h)\|_{L^p(K)}\leq C\|f\|_{L^p({\BBR}^{n-1})}. 
$$
The constant $C$ in the estimate above is independent of $K$ and $k$, 
so taking the supremum in each of $k$ and $K$ gives the desired estimate for $u$ on $\Omega={\BBR}^n$.

The $L^p$ solvability in an interval around the value $p=2$ is established later, in Section~6.
\end{proof}


\section{Estimates for the square function $S(u)$ of a solution}
\label{S4}

In this section we establish a one-sided estimate of the square-function in terms of boundary data and 
the nontangential maximal function. Fix some $h>1$, and recall the infinite strip $\Omega^h$ defined above.
Also, let $u_h$ be the energy solution to \eqref{E:D-energy}, extended to be zero above the height $h$.  
Due to the reductions we have made, it suffices to work with a coefficient tensor satisfying $A_{00}=I_{N\times N}$.


%
\begin{lemma}\label{S3:L4}
Let $u_h:\Omega={\mathbb{R}}^n_{+}\to{\BBR}^N$ be as above with the Dirichlet boundary datum $f\in L^{2}(\partial\Omega;{\BBR}^N)$. 
Assume that $A$ is strongly elliptic, satisfies $A_{0j}^{\alpha\beta}=\delta_{\alpha\beta}\delta_{0j}$,
and the measure $\mu$ defined as in \eqref{Car_hatAA} is Carleson. Then there exists a finite constant 
$C=C(n,N,\lambda,\Lambda)>0$ such that for all $r>0$ one has
\begin{align}\label{S3:L4:E00}
&\lambda\iint_{[0,r/2]\times\partial\Omega}|\nabla u_h|^{2}x_0\,dx'\,d x_0 
+\frac{2}{r}\iint_{[0,r]\times\partial \Omega} |u_h(x_0,x')|^{2}\,dx'\,dx_0 
\\[4pt]
&\hskip 0.20in
\leq\int_{\partial\Omega}|u_h(0,x')|^{2}\,dx' 
+\int_{\partial\Omega}|u_h(r,x')|^{2}\,dx'
+C\|\mu\|_{\mathcal{C}}\int_{\partial\Omega}\left[\tilde{N}^{r}_a(u_h)\right]^{2}\,dx'.
\nonumber
\end{align}
\end{lemma}

\begin{proof}  In the following proof we drop the dependance of $u_h$ on the parameter $h$ and just write $u=u_h$. 
Fix an arbitrary $y'\in\partial\Omega\equiv{\mathbb{R}}^{n-1}$ and consider first the case when $r\le h$. 
Pick a smooth cutoff function $\zeta$ which is $x_0-$independent and satisfies
\begin{equation}\label{cutoff-F}
\zeta= 
\begin{cases}
1 & \text{ in } B_{r}(y'), 
\\
0 & \text{ outside } B_{2r}(y').
\end{cases}
\end{equation}
Moreover, assume that $|(\nabla \zeta)(y')|\leq c/r$ for some positive constant $c$ independent of $y'$. 
We begin by considering the integral quantity 
\begin{equation}\label{A00}
\mathcal{I}:=\iint_{[0,r]\times B_{2r}(y')}A_{ij}^{\alpha\beta}\partial_{j}u_{\beta} 
\partial_{i}u_{\alpha}x_0\zeta\,dx'\,dx_0
\end{equation}
with the usual summation convention understood. In relation to this we note that the uniform 
ellipticity \eqref{EllipA} gives 
\begin{equation}\label{cutoff-AA}
\mathcal{I}\geq{\lambda}\iint_{[0,r]\times B_{2r}}\sum_{\alpha}|\nabla u_{\alpha}|^2 x_0\zeta\,dx'\,dx_0
={\lambda}\iint_{[0,r]\times B_{2r}}|\nabla u|^2 x_0\zeta\,dx'\,dx_0,
\end{equation}
where we agree henceforth to abbreviate $B_{2r}:=B_{2r}(y')$ whenever convenient. 
The idea now is to integrate by parts the formula for $\mathcal I$ in order 
to relocate the $\partial_i$ derivative. This gives 
\begin{align}\label{I+...+IV}
\mathcal{I}
&= \int_{\partial\left[(0,r)\times B_{2r}\right]} 
A_{ij}^{\alpha\beta}\partial_{j}u_{\beta}u_{\alpha}x_0\zeta\nu_{x_i}\,d\sigma 
\nonumber\\[4pt]
&\quad -\iint_{[0,r]\times B_{2r}}\partial_{i}\left(A_{ij}^{\alpha\beta} 
\partial_{j}u_{\beta}\right)u_{\alpha}x_0\zeta\,dx'\,dx_0 
\nonumber\\[4pt]
&\quad -\iint_{[0,r]\times B_{2r}}A_{ij}^{\alpha\beta}\partial_{j}u_{\beta}u_{\alpha}\partial_{i}x_0\zeta\,dx'\,dx_0 
\nonumber\\[4pt]
&\quad -\iint_{[0,r]\times B_{2r}}A_{ij}^{\alpha\beta}\partial_{j}u_{\beta}u_{\alpha}x_0\partial_{i}\zeta\,dx'\,dx_0
\nonumber\\[4pt]
&=:I+II+III+IV,
\end{align}
where $\nu$ is the outer unit normal vector to $(0,r)\times B_{2r}(y')$. 
Bearing in mind that $A_{0j}^{\alpha\beta}=0$ for $j>0$, and upon recalling that we 
are assuming $A_{00}=I_{N\times N}$, the boundary term $I$ simply becomes
\begin{equation}\label{cutoff-BBB}
I=\int_{B_{2r}}\partial_{0}u_{\beta}(r,x')u_{\beta}(r,x')\,r\,\zeta\,dx'.
\end{equation}

As $u$ is a weak solution of  $\mathcal L u=0$ in $[0,r]\times B_{2r}$, we use this PDE to transform $II$ into
\begin{equation}\label{cutoff-CCC}
II=\iint_{[0,r]\times B_{2r}}B_{i}^{\alpha\beta}(\partial_{i}u_{\beta})u_{\alpha}x_0\zeta\,dx'\,dx_0.
\end{equation}
To further estimate this term we use Cauchy-Schwarz inequality, the Carleson condition for $B$, 
and Proposition~\ref{T:Car} in order to write
\begin{align}\label{TWO-TWO}
II &\leq\left(\iint_{[0,r]\times B_{2r}}\left(B_i^{\alpha\beta}\right)^{2} 
|u_{\alpha}|^{2} x_0\zeta\,dx'\,dx_0\right)^{1/2}  
\cdot\left(\iint_{[0,r]\times B_{2r}}|\partial_{j}u_{\beta}|^{2}x_0\zeta\,dx'\,dx_0\right)^{1/2} 
\nonumber\\[4pt]
&\leq C(\lambda,\Lambda,N)\left(\|\mu\|_{\mathcal{C}}\int_{B_{2r}} 
\left[\tilde{N}^r_a(u)\right]^{2}\,dx'\right)^{1/2}\cdot\mathcal{I}^{1/2}. 
\end{align}

As $\partial_ix_0=0$ for $i>0$, the term $III$ is non-vanishing only for $i=0$. We further split this term
by considering the cases when $j=0$ and $j>0$, respectively. When $j=0$, we use that $A_{00}^{\alpha\beta}=I_{N\times N}$. 
This yields 
\begin{align}\label{u6fF}
III_{\{j=0\}} &=-\frac{1}{2}\iint_{[0,r]\times B_{2r}} 
\sum_\beta\partial_{0}\left(u_{\beta}^{2}\zeta\right)\,dx'\,dx_0 
\nonumber\\[4pt]
&=-\frac{1}{2}\int_{B_{2r}}\sum_\beta u_{\beta}(r,x')^{2}\zeta\,dx'
+\frac{1}{2}\int_{B_{2r}}\sum_\beta u_{\beta}(0,x')^{2}\zeta\,dx'.
\end{align}
Corresponding to $j>0$ we simply recall that $A_{0j}^{\alpha\beta}=0$ for $j>0$ to conclude that $III_{\{j>0\}}=0$. 

Adding up all terms produced so far we obtain
\begin{align}\label{square01}
\mathcal{I} &\leq\int_{B_{2r}}\partial_{0}u_{\beta}(r,x')u_{\beta}(r,x')\,r\,\zeta\,dx' 
\nonumber\\[4pt]
&\quad-\frac{1}{2}\int_{B_{2r}}\sum_\beta u_{\beta}(r,x')^{2}\zeta\,dx'
+\frac{1}{2}\int_{B_{2r}}\sum_\beta u_{\beta}(0,x')^{2}\zeta\,dx' 
\nonumber\\[4pt]
&\quad +C(\lambda,\Lambda,n,N)\|\mu\|_{\mathcal{C}}\int_{B_{2r}}\left[\tilde{N}^{r}_a(u)\right]^2\,dx' 
+\frac12\mathcal{I}+IV,
\end{align}
where we have used the arithmetic-geometric inequality for the expression bounding the term $II$ in \eqref{TWO-TWO}.

To obtain a global version of \eqref{square01}, the idea is to consider a sequence of disjoint boundary balls 
$(B_r(y'_k))_{k\in\mathbb N}$ such that $\bigcup_{k}B_{2r}(y'_k)$ covers $\partial\Omega={\BBR}^{n-1}$ 
and bring in a partition of unity $(\zeta_{k})_{k\in\mathbb N}$ subordinate to this cover. That is, 
assume $\sum_k \zeta_{k}=1$ on ${\BBR}^{n-1}$ and each $\zeta_{k}$ is supported in $B_{2r}(y'_k)$. 
Write $IV_k$ for each term as the last expression in \eqref{I+...+IV} corresponding to 
$B_{2r}=B_{2r}(y'_k)$. Given that $\sum_k \partial_i\zeta_{k} = 0$ for each $i$, by summing 
\eqref{square01} over all $k$'s gives $\sum_{k} IV_k= 0$. It follows that
\begin{align}\label{square02}
&\hskip -0.20in 
\frac\lambda{2}\iint_{[0,r]\times{\BBR}^{n-1}}|\nabla u|^2\,x_0\,dx'\,dx_0 
\nonumber\\[4pt]
&\hskip 0.20in
\leq\int_{{\BBR}^{n-1}}\partial_{0}u_{\beta}(r,x')u_{\beta}(r,x')\,r\,dx' 
\nonumber\\[4pt]
&\hskip 0.20in
\quad -\frac{1}{2}\int_{{\BBR}^{n-1}}\sum_\beta u_{\beta}(r,x')^{2}\,dx'
+\frac{1}{2}\int_{{\BBR}^{n-1}}\sum_\beta u_{\beta}(0,x')^{2}\,dx' 
\nonumber\\[4pt]
&\quad +C\|\mu\|_{\mathcal{C}}\int_{{\BBR}^{n-1}}\left[\tilde{N}^{r}_a(u)\right]^2\,dx'.
\end{align}
We have therefore established \eqref{square02} for $r\leq h$. There remains to 
observe that \eqref{square02} also holds for $r>h$, by virtue of the fact that $u=0$ when $r>h$.
From this, \eqref{S3:L4:E00} follows by integrating \eqref{square02} in $r$ over $[0,r']$ and then dividing by $r'$.
\end{proof}

Lemma~\ref{S3:L4} has three important corollaries. 

\begin{corollary}\label{S4:C1} 
Retain the assumptions of Lemma~\ref{S3:L4}. Then, given a weak solution $u_h$ of \eqref{E:D-strip}, 
for any $r>0$ we have
\begin{equation}\label{S3:L4:E00aa}
\lambda\iint_{[0,r/2]\times\partial\Omega}|\nabla u_h|^{2}x_0\,dx'\,d x_0 
\leq C(1+\|\mu\|_{\mathcal{C}})\int_{\partial\Omega}\left[\tilde{N}^{r}_a(u_h)\right]^{2}\,dx'.
\end{equation}
That is, $\|S^{r/2}_a(u_h)\|_{L^2(\partial \Omega)}\leq C\|\tilde{N}^r_a(u_h)\|_{L^2(\partial \Omega)}$ 
with the intervening constant depending only on $\lambda,\,\Lambda,n,N,a$, and $\|\mu\|_{\mathcal C}$. 
In particular, letting $r\to\infty$ yields a version for the global square and nontangential 
maximal functions, namely 
\begin{equation}\label{Eqqq-3}
\|S_a(u_h)\|_{L^2(\partial\Omega)}\leq C\|\tilde{N}_a(u_h)\|_{L^2(\partial \Omega)},
\end{equation}
for all energy solutions $u_h$ of \eqref{E:D-strip}.
\end{corollary}

\begin{corollary}\label{S4:C2} 
Under the assumptions of Lemma \ref{S3:L4} for any energy solution $u_h$ of \eqref{E:D-strip} we have
\begin{align}\label{S3:L4:E00bb}
\lambda\iint_{\BBR^n_+}|\nabla u_h|^{2}x_0\,dx'\,d x_0 
&\leq\int_{\BBR^{n-1}}|u_h(0,x')|^{2}\,dx'
\nonumber\\[6pt]
&\quad+C\|\mu\|_{\mathcal{C}}\int_{\BBR^{n-1}}\left[\tilde{N}_a(u_h)\right]^{2}\,dx'.
\end{align}
\end{corollary}

\begin{corollary}\label{S5:C3} 
Under the assumptions of Lemma \ref{S3:L4}, for any energy solution $u_h$ of \eqref{E:D-strip} we have, 
for any $x'\in\mathbb R^{n-1}$ and $r>0$, 
\begin{align}\label{eq5.15}
&\iint_{[0,r/2]\times B_r}|\nabla u_h|^{2}x_0\,dx'\,dx_0  
\nonumber\\[6pt]
&\leq C\left[\int_{B_{2r}}|u_h(0,x')|^{2}\,dx' 
+\int_{B_{2r}}|u_h(r,x')|^{2}\,dx'
+\|\mu\|_{\mathcal{C}}\int_{B_{2r}}\left[\tilde{N}^r_{a}(u_h)\right]^{2}\,dx'\right]
\nonumber\\[6pt]
&\leq C(2+\|\mu\|_{\mathcal{C}})\int_{B_{2r}}\left[\tilde{N}^{2r}_{a}(u_h)\right]^{2}\,dx'.
\end{align}
\end{corollary}

This is a local version of the Corollary \ref{S4:C2}. Its justification may be carried out exactly as in 
the proof above, up to \eqref{square01}. Then, instead of summing over different balls $B_r$ covering $\mathbb R^{n-1}$, 
we now estimate the terms in $IV$. Both of these terms are of the same type and each may be bounded (up to a multiplicative 
constant) by
\begin{equation}\label{eq5.16}
\iint_{[0,r]\times B_{2r}}|\nabla u||u|x_0|\partial_T\zeta|dx'dx_0,
\end{equation}
where $\partial_T\zeta$ denotes any of the derivatives in the direction parallel 
to the boundary and $u=u_h$ as before. Recall that $\zeta$ is a smooth cutoff function equal to $1$ 
on $B_r$ and $0$ outside $B_{2r}$. In particular, we may assume $\zeta$ to be of the form $\zeta=\eta^2$ 
for some smooth function $\eta$ such that $|\nabla_T\eta|\le C/r$. By Cauchy-Schwarz's inequality, the term 
\eqref{eq5.16} may be further majorized by
\begin{align}\label{eq5.17}
\Bigg(\iint_{[0,r]\times B_{2r}}& |\nabla u|^2x_0(\eta)^2dx'dx_0\Bigg)^{1/2}\cdot
\left(\iint_{[0,r]\times B_{2r}}|u|^{2}x_0|\nabla_T\eta|^2dx'dx_0\right)^{1/2}
\nonumber\\[6pt]
&\lesssim{\mathcal I}^{1/2}\left(\frac1r\iint_{[0,r]\times B_{2r}}|u|^2dx'dx_0\right)^{1/2}
\nonumber\\[6pt]
&\le \varepsilon{\mathcal I}+C_\varepsilon\int_{B_{2r}}\left[\tilde{N}^r_{a}(u)\right]^{2}\,dx'.
\end{align}
In the last step we have used arithmetic-geometric inequality and a trivial estimate of the solid integral $|u|^2$ by the 
averaged nontangential maximal function. Substituting \eqref{eq5.17} into \eqref{square01} the estimate 
\eqref{eq5.15} follows by integrating in $r$ over $[0,r']$ and dividing by $r'$ exactly as done above.

\begin{corollary}\label{S5:C4} 
Retain the assumptions of Lemma \ref{S3:L4}, and fix $p>0$ and $a>0$.
Then there exists a finite constant $C=C(n,N,\lambda,\Lambda,p,a,\|\mu\|_{\mathcal C})>0$ such that 
\begin{equation}\label{S3:C7:E00ooXX}
\|S_a(u_h)\|_{L^{p}({\BBR}^{n-1})}\le C\|\tilde{N}_a(u_h)\|_{L^{p}({\BBR}^{n-1})}
\end{equation}
for any energy solution $u_h$ of \eqref{E:D-strip}.
\end{corollary}

This is a consequence of Corollary~\ref{S5:C3}, following \cite{FSt}
(see a more detailed discussion in the proof of Proposition~\ref{S3:C7}).

\section{Bounds for the nontangential maximal function by the square function}
\label{SS:43}

As before, we shall work under the assumption that $\Omega=\BBR^n_+$. We will only assume the 
Legendre-Hadamard condition \eqref{EllipLH} and merely impose large Carleson conditions on the coefficients. 
(Starting with a graph Lipschitz domain, we may always reduce matters to this setting via the pull-back map $\rho$).

Our aim in this section is to establish a reverse version of the inequality in Corollary~\ref{S4:C1}. 
The approach necessarily differs from the usual argument in the scalar elliptic case due to the fact 
that certain estimates, such as interior H\"older regularity of a weak solution, are unavailable for 
the class of systems presently considered. Hence, alternative arguments bypassing such difficulties 
must be devised. 

The major innovation is the use of an entire family of Lipschitz graphs on which the nontangential 
maximal function is large in lieu of a single graph constructed via a stopping time argument. 
This is necessary as we are using $L^2$ averages of solutions to define the nontangential maximal 
function and hence the knowledge of certain bounds for a solution on a single graph provides no 
information about the $L^2$ averages over interior balls.

The energy solutions $u_h$ constructed using Lax-Milgram lemma on $\Omega_h$ and extended by zero on 
$\{(x_0,x'):\, x_0>h\}$ a priori belong to the space $\dot{W}^{1,2}(\mathbb R^n_+;\mathbb R^N)$.
Since $u_h(h,\cdot)=0$, this implies $u_h\in W^{1,2}(\mathbb R^n_+;\mathbb R^N)$ (with norm depending of $h$). 
We drop dependence on $h$ for now and simply write $u=u_h$. For the function $w$ defined in $\Omega$ as 
in \eqref{w}, and a constant $\nu>0$, define the set
\begin{equation}\label{E}
E_{\nu,a}:=\big\{x'\in\partial\Omega:\,N_{a}(w)(x')>\nu\big\}
\end{equation}
where, as usual, $a>0$ is a fixed background parameter. Also, 
consider the map $\hbar:\partial\Omega\to\BBR$ given at each $x'\in\partial\Omega$ by 
\begin{equation}\label{h}
\hbar_{\nu,a}(w)(x'):=\inf\left\{x_0>0:\,\sup_{z\in\Gamma_{a}(x_0,x')}w(z)<\nu\right\}
\end{equation}
with the convention that $\inf\varnothing=\infty$. 
We remark that $\hbar$ differs from the function $\tilde{\hbar}:\partial\Omega\to\BBR$ defined 
at each $x'\in\partial\Omega$ as
\begin{equation}\label{Eqqq-4}
\tilde{\hbar}_{\nu, a}(w)(x'):=\sup\left\{x_0 >0:\,\sup_{z\in\Gamma_{a}(x_0,x')}w(z)>\nu\right\}.
\end{equation}
The function $\tilde{\hbar}$ has been used in arguments for scalar equations 
(cf. \cite[pp.\,212]{KP01} and \cite{KKPT}). While there are clear similarities in the manner 
in which the functions $\hbar$ and $\tilde{\hbar}$ are defined, throughout this paper we prefer to use $\hbar$ 
as it works better for elliptic systems. 

At this point we observer that $\hbar_{\nu,a}(w,x')<\infty$ for all points $x'\in\partial\Omega$. 
This is due to the fact that the function $u$ vanishes above height $h$, hence the averages $w$ 
vanish above the height $2h$. It follows that $\hbar_{\nu,a}(w)(x')<\infty$ and, in fact, $\hbar_{\nu,a}(w)(x')<2h$.

\begin{lemma}\label{S3:L5}
Let $u$ be an energy solution of \eqref{E:D-strip}, and associated with it the function $w$ as in \eqref{w}. 
Also, fix two positive numbers $\nu,a$. Then the following properties hold.
\vglue2mm

\noindent (i)
The function $\hbar_{\nu, a}(w)$ is Lipschitz, with a Lipschitz constant $1/a$. That is,
\begin{equation}\label{Eqqq-5}
\left|\hbar_{\nu,a}(w)(x')-\hbar_{\nu,a}(w)(y')\right|\leq a^{-1}|x'-y'|
\end{equation}
for all $x',y'\in\partial\Omega$.

\vglue2mm

\noindent (ii)
Given an arbitrary $x'\in E_{\nu,a}$, let $x_0:=\hbar_{\nu,a}(w)(x')$. Then there exists a 
point $y=(y_0,y')\in\partial\Gamma_{a}(x_0,x')$ such that $w(y)=\nu$ and $\hbar_{\nu,a}(w)(y')=y_0$. 		
\end{lemma}

\begin{proof} 
To prove the claim formulated in part {\it (i)}, pick a pair of arbitrary points 
$x',y'\in\partial\Omega$ and set $y_0:=\hbar_{\nu,a}(w)(y')$, $x_0:=\hbar_{\nu,a}(w)(x')$. 
Without loss of generality it may be assumed that $y_0<x_0$. In particular, this 
forces $x_0\in(0,\infty)$. Seeking a contradiction, suppose
\begin{equation}\label{contr}
|x'-y'|<a(x_0-y_0).
\end{equation}
Then simple geometric considerations give
\begin{equation}\label{Eqqq-6}
\overline{\Gamma_{a}(x_0, x')}\subset \Gamma_{a}(y_0, y').
\end{equation}
In particular, there exists $\varepsilon\in(0,2x_0)$ with the property that
\begin{equation}\label{Eqqq-7}
\overline{\Gamma_{a}(x_0, x')}\subset \Gamma_{a}(y_0+\varepsilon, y').
\end{equation}
Hence,
\begin{equation}\label{Eqqq-8}
\overline{\Gamma_{a}(x_0-\varepsilon/2,x')}\subset\Gamma_{a}(y_0+\varepsilon/2,y').
\end{equation}
It follows that 
\begin{equation}\label{Eqqq-9}
\sup_{\Gamma_{a}(x_0-\varepsilon/2,x')}w\leq\sup_{\Gamma_{a}(y_0+\varepsilon/2,y')}w<\nu,
\end{equation}
the last inequality being true by the definition of $y_0=\hbar_{\nu,a}(w)(y')$ in \eqref{h}. 
This however implies that
\begin{equation}\label{Eqqq-10}
x_0-\varepsilon/2\geq \hbar_{\nu,a}(x')=x_0,
\end{equation}
which is the desired contradiction. Therefore the assumption made in \eqref{contr} is false 
which then entails $0<a(x_0-y_0)\leq|x'-y'|$. From this the claim in part {\it (i)} follows.

To justify the claim recorded in part {\it (ii)}, fix some $x'\in E_{\nu,a}$ and note that this implies
$x_0=\hbar_{\nu,a}(w)(x')>0$. To show that there exists $y=(y_0,y')\in\partial\Gamma_{a}(x_0,x')$ such that 
$w(y)=\nu$ we employ a compactness argument. Due to the decay of $w$ at infinity it follows that for 
a sufficiently large $r$ (depending on $u$) we have 
\begin{equation}\label{Eqqq-11}
\sup_{\{z\in{\mathbb{R}}^n_{+}:\,z_0\geq r\}}w(z)\leq\nu/2.
\end{equation}
If it were true that $w(z)<\nu$ for all points $z\in\partial\Gamma_{a}(x_0,x')\cap\{z_0\leq r\}$ then, 
as the function $w$ is continuous, each such point $z$ would posses a neighborhood ${\mathcal{O}}_z$ where $w<\nu$.
The family $\big\{{\mathcal{O}}_z\big\}_z$ then constitutes an open cover of the compact set 
$\partial\Gamma_{a}(x_0,x')\cap\{z_0\leq r\}$ and may therefore be refined to a finite sub-cover, say  
$\big\{{\mathcal{O}}_{z_i}\big\}_{1\leq i\leq k}$. Upon introducing 
\begin{equation}\label{Eqqq-12}
\mathcal{S}:=\Big(\bigcup_{i=1}^k{\mathcal{O}}_{z_i}\Big)\cup\Gamma_a(x_0,x')\cup\{(z_0,z'):\,z_0\geq r\}
\end{equation}
it follows that 
\begin{equation}\label{Eqqq-13}
w(z)<\nu,\qquad\forall\,z\in{\mathcal{S}}.
\end{equation}
However, for some small $\varepsilon\in(0,x_0)$
\begin{equation}\label{Eqqq-14}
\overline{\Gamma_a(x_0-\varepsilon,x')}\subset\mathcal S,
\end{equation}
and the compactness of the set $\mathcal{S}\cap\{z_0\leq r\}$ together with decay of $w$ above height $r$
entail 
\begin{equation}\label{Eqqq-15}
\sup_{z\in\Gamma_a(x_0-\varepsilon,x')}w(z)<\nu.
\end{equation}
This contradicts the definition of $x_0=\hbar_{\nu,a}(w)(x')$ in \eqref{h}. Bearing in mind the definition 
of $\hbar_{\nu,a}(w)(x')$ and the continuity of $w$, we conclude that for some point  
$y=(y_0,y')\in\partial\Gamma_{a}(x_0,x')$ we have $w(y)=\nu$. In turn, this forces
$\hbar_{\nu,a}(w)(y')\geq y_0$. On the other hand, since $\Gamma_a(y_0,y')\subseteq\Gamma_a(x_0,x')$, we have 
\begin{equation}\label{Eqqq-16}
\Gamma_a(y_0+\varepsilon,y')\subset\Gamma_a(x_0+\varepsilon,x')\,\,\text{ for every }\,\,\varepsilon>0
\end{equation}
which implies
\begin{equation}\label{Eqqq-16a}
\sup_{\Gamma_a(y_0+\varepsilon,y')}w\leq\sup_{\Gamma_a(x_0+\varepsilon,x')}w<\nu
\,\,\text{ for every }\,\,\varepsilon>0.
\end{equation}
In turn, this allows us to conclude that $\hbar_{\nu, a}(w)(y')\leq y_0+\varepsilon$ 
for every $\varepsilon>0$. Hence, ultimately it follows that $\hbar_{\nu,a}(w)(y')=y_0$, as claimed.
\end{proof}

\begin{lemma}\label{l6} 
Assume as before that $u$ is an energy solution of the system \eqref{E:D-strip} in $\Omega=\BBR^n_+$. 
Then for any $a>0$ there exists $b=b(a)>a$ and $\gamma=\gamma(a)>0$ such that the following holds. 
Having fixed an arbitrary $\nu>0$, for each point $x'$ from the set 
\begin{equation}\label{Eqqq-17}
\big\{x':\,N_{a}(w)(x')>\nu\mbox{ and }S_{b}(u)(x')\leq\gamma\nu\big\}
\end{equation}
there exists a boundary ball $R$ with $x'\in 2R$ and such that
\begin{equation}\label{Eqqq-18}
\big|w\big(\hbar_{\nu,a}(w)(z'),z'\big)\big|>\nu/{2}\,\,\text{ for all }\,\,z'\in R.
\end{equation}
\end{lemma}

\begin{proof} Let $x'\in\partial\Omega$ be such that $N_{a}(w)(x')>\nu$ and 
$S_{b}(u)(x')\leq\gamma\nu$. As before, set $x_0:=\hbar_{\nu, a}(w,x')$.  
From part {\it (ii)} in Lemma~\ref{S3:L5} we know that
there exists a point $y=(y_0,y')\in\partial\Gamma_{a}(x_0,x')$ such that $w(y)=\nu$. 
Let $d:=|x'-y'|$ and define $R=\{z'\in\partial\Omega:\,|z'-y'|<3ay_0/2\}$. 
Then $x'\in 2R$ since $d<ay_0$. This choice
also guarantees that $\hbar_{\nu,a}(w,z')\in [y_0/3,5y_0/3]$ by {\it (i)} in Lemma~\ref{S3:L5}. 

To proceed, consider the set
\begin{equation}\label{Eqqq-19}
\mathcal{O}:=\big\{z=(z_0,z')\in\Omega:\,z'\in R\,\mbox{ and }\,z_0\in [y_0/3,5y_0/3]\big\}.
\end{equation}
In particular, $y\in{\mathcal{O}}$. 
Then all claims in the current lemma are justified as soon as we establish that 
\begin{equation}\label{uHBB}
w(z)>\nu/{2}\,\,\text{ for all }\,\,z\in\mathcal O.
\end{equation}
With this goal in mind, consider $\bigcup_{z\in\mathcal{O}}B_{z_0/2}(z)$. All points of this 
set are at least $y_0/6$ away from the boundary of $\Omega$ and the diameter of this set 
is comparable to $y_0$. Select the number $b>a$ so that
\begin{equation}\label{Eqqq-20}
{\mathcal{B}}:=\bigcup_{z\in\mathcal{O}}B_{z_0/2}(z)\subset\Gamma_b(0,x').
\end{equation}
A simple geometrical argument shows that $b$ can be chosen independently of the location of points 
$x',y'$, and only depends on the size of $a$. Our goal is to estimate the difference $|w(z)-w(y)|$ 
for all $z\in\mathcal{O}$. To this end, fix some $z\in\mathcal{O}$. Abbreviating $B:=B_{1/2}(0)$ 
then permits us to express
\begin{equation}\label{Eqqq-21}
w(z)=\left(\dint_{B}|u(z+z_0\xi)|^{2}\,d\xi\right)^{1/2},\quad
w(y)=\left(\dint_{B}|u(y+y_0\xi)|^{2}\,d\xi\right)^{1/2}.
\end{equation}
It follows that 
\begin{align}
w(z) &=\left(\dint_{B}\left|u(y+y_0\xi)+[u(z+z_0\xi)-u(y+y_0\xi)]\right|^{2}\,d\xi\right)^{1/2}
\nonumber\\[4pt]
&\leq\left(\dint_{B}\left|u(y+y_0\xi)\right|^{2}\,d\xi\right)^{1/2}
+\left(\dint_{B}\left|u(z+z_0\xi)-u(y+y_0\xi)\right|^{2}\,d\xi\right)^{1/2}
\nonumber\\[4pt]
&=w(y)+\left(\dint_{B}\left|u(z+z_0\xi)-u(y+y_0\xi)\right|^{2}\,d\xi\right)^{1/2}.
\end{align}
Since a similar estimate holds when the roles of $y$ and $z$ are interchanged, we eventually conclude that
\begin{equation}\label{Daver}
|w(z)-w(y)|^2\leq\dint_{B}\left|u(z+z_0\xi)-u(y+y_0\xi)\right|^{2}\,d\xi.
\end{equation}

Going further, the Fundamental Theorem of Calculus gives that for any 
two points $z_1,z_2\in\mathcal{B}$ we have
\begin{align}\label{eqFC}
\left|u(z_1)-u(z_2)\right|^2
& \leq\left|\int^1_0(\nabla u)\big(z_1+(z_2-z_1)\tau\big)\cdot(z_1-z_2)\,d\tau\right|^2
\nonumber\\[4pt]
& \leq|z_1-z_2|^2\int^1_0\big|(\nabla u)\big(z_1+(z_2-z_1)\tau\big)\big|^2\,d\tau
\nonumber\\[4pt]
& =y_0^{n-2}|z_1-z_2|^2\int^1_0\big|(\nabla u)\big(z_1+(z_2-z_1)\tau\big)\big|^2 y_0^{2-n}\,d\tau
\nonumber\\
& \leq Cy_0^{n-1}\int_{[z_1,z_2]}|(\nabla u)(q)|^2 q_0^{2-n}\,ds(q),
\end{align}
where the last integral is understood as a line integral over the segment joining $z_1$ and $z_2$. 
We have also use the fact that $|z_1-z_2|\leq Cy_0$ for all $z_1,z_2\in{\mathcal{B}}$. 
We apply this formula to generic pairs of points of the form $z+z_0\xi$, $y+y_0\xi$ for 
$z\in{\mathcal{O}}$ and $\xi\in B$ (which, by design, are in ${\mathcal{B}}$) 
and then integrate in $\xi$. Notice that, for various points $\xi$, the lines joining 
$z+z_0\xi$ with $y+y_0\xi$ are almost parallel; in fact they are genuinely parallel when $z_0=y_0$. 
When integrating in $\xi$ over $B$ a typical point $q$ in the very last expression in  
\eqref{eqFC} considered with $z_1:=z+z_0\xi$ and $z_2:=y+y_0\xi$ 
will belong to certain line segments joining these points with $\xi$ belonging to a certain 
subset of $B$ of $1$-dimensional Hausdorff measure, having size $O(1)$ relative to this measure. 
Hence, 
\begin{equation}\label{Daver2}
\frac1{|B|}\int_{B}\left|u(z+z_0\xi)-u(y+y_0\xi)\right|^2\,d\xi
\leq C\int_{\mathcal{H}}|(\nabla u)(q)|^2 q_0^{2-n}\,dq,
\end{equation}
where $\mathcal{H}$ denotes the convex hull of the set $B_{z_0/2}(z)\cup B_{y_0/2}(y)\subset\Gamma_{b}(0,x')$,
which is a set of diameter comparable to $y_0$. The factor $y_0^{n-1}$ in \eqref{eqFC} disappears 
after integrating in $\xi$ due to the natural change of variables which takes $ds(q)d\xi$ into 
$dq$ in \eqref{Daver2}, the natural Lebesgue measure on $\mathcal{H}$. 
Because $\mathcal{H}$ is contained in $\Gamma_{b}(0,x')$ the right-hand side of \eqref{Daver2} 
may be further estimated by $S^2_{b}(u)(x')\leq\gamma^2\nu^2$. Hence, by combining 
\eqref{Daver}-\eqref{Daver2} we obtain 
\begin{equation}\label{Daver3}
|w(z)-w(y)|^2\leq C(a,n,N)(\gamma\nu)^2\leq\frac{\nu^2}4,
\end{equation} 
if $\gamma$ is chosen so that $C(a,n,N)\gamma^2<1/4$. It follows that for any $z\in\mathcal O$ we have
\begin{equation}\label{Eqqq-22}
w(z)\ge w(y)-|w(y)-w(z)|\ge \nu-\frac{\nu}2=\frac{\nu}2.
\end{equation}
Hence the claim in \eqref{uHBB} follows, finishing the proof of the lemma. 
\end{proof}

Given a Lipschitz function $\hbar:{\mathbb{R}}^{n-1}\to{\mathbb{R}}$, denote by $M_\hbar$ the 
Hardy-Littlewood maximal function considered on the graph of $\hbar$. That is, 
given any locally integrable function $f$ on the Lipschitz surface 
$\Lambda_\hbar=\{(\hbar(z'),z'):\,z'\in\BBR^{n-1}\}$, define 
$(M_\hbar f)(x):=\sup_{r>0}\dint_{\Lambda_\hbar\cap B_r(x)}|f|\,d\sigma$ for each $x\in\Lambda_\hbar$. 

\begin{corollary}\label{S3:L6} 
Let $u$ is an energy solution of the system \eqref{E:D-strip} in $\Omega=\BBR^n_+$ and fix $a>0$. 
Associated with these, let $b,\,\gamma$ be as in Lemma~\ref{l6}. Then there exists a finite 
constant $C=C(n)>0$ with the property that for any $\nu>0$ and any point $x'\in E_{\nu,a}$ 
such that $S_{b}(u)(x')\leq\gamma\nu$ one has
\begin{equation}\label{Eqqq-23}
(M_{\hbar_{\nu,a}}w)\big(\hbar_{\nu,a}(x'),x'\big)\geq\,C\nu.
\end{equation}
\end{corollary}

\begin{proof} 
Fix a point $x'\in E_{\nu,a}$ where $S_{b}(u)(x')\leq\gamma\nu$. Lemma~\ref{l6} then guarantees the 
existence of a boundary ball $R$ with the property that $w(\hbar_{\nu,a}(w)(z'),z')>\nu/{2}$ for all 
$z'\in R$ and $x'\in 2R$. Granted this, it follows that 
\begin{equation}\label{Eqqq-24}
(M_{\hbar_{\nu,a}}w)\big(\hbar_{\nu,a}(w)(x'),x'\big)\geq\frac1{|2R|}\int_R w\big(\hbar_{\nu,a}(w)(z'),z'\big)\,dz'
\geq\frac{|R|}{|2R|}\frac{\nu}{{2}},
\end{equation}
as desired. 
\end{proof}

\begin{lemma}\label{S3:L8} 
Consider the system \eqref{E:D} with coefficients satisfying Carleson condition and the condition \eqref{EllipLH}. 
Then there exists $a>0$ with the following significance. Suppose $u$ is a weak solution of  \eqref{E:D-strip} 
in $\Omega={\mathbb{R}}^n_{+}$. Select $\theta\in[1/6,6]$ and, having picked $\nu>0$ arbitrary,  
let $\hbar_{\nu,a}(w)$ be as in \eqref{h}. Also, consider the domain 
$\mathcal{O}=\{(x_0,x')\in\Omega:\,x_0>\theta \hbar_{\nu,a}(x')\}$ with boundary  
$\partial\mathcal{O}=\{(x_0,x')\in\Omega:\,x_0=\theta \hbar_{\nu,a}(x')\}$. In this context, 
for any surface ball $\Delta_r=B_r(Q)\cap\partial\Omega$, with $Q\in\partial\Omega$ and $r>0$ 
chosen such that $\hbar_{\nu,a}(w)\leq 2r$ pointwise on $\Delta_{2r}$, 
one has
\begin{align}\label{TTBBMM}
\int_{\Delta_r}\big|u\big(\theta \hbar_{\nu,a}(w)(\cdot),\cdot\big)\big|^2\,dx' 
&\leq C(1+\|\mu\|^{1/2}_{\mathcal C})\|S_b(u)\|_{L^2(\Delta_{2r})}
\|\tilde{N}_a(u)\|_{L^2(\Delta_{2r})}
\nonumber\\
&\quad+C\|S_b(u)\|^2_{L^2(\Delta_{2r})}+\frac{c}{r}\iint_{\mathcal{K}}|u(X)|^{2}\,dX.
\end{align}
Here $C=C(\lambda,\Lambda,n,N)\in(0,\infty)$ and $\mathcal{K}$ is a region inside $\mathcal{O}$ of diameter, 
distance to the boundary $\partial\mathcal{O}$, and distance to $Q$, are all comparable to $r$. 
Also, the parameter $b>a$ is as in Lemma~\ref{l6}, and the cones used to define the square and nontangential 
maximal functions in this lemma have vertices on $\partial\Omega$.

Moreover, the term $\displaystyle\iint_{\mathcal{K}}|u(X)|^2\,dX$ appearing 
in \eqref{TTBBMM} may be replaced by the quantity
\begin{equation}\label{Eqqq-25}
Cr^{n-1}|u_{\rm av}(A_r)|^2+C\int_{\Delta_{2r}}S^2_b(u)\,d\sigma,
\end{equation}
where $A_r$ is any point inside $\mathcal{K}$ {\rm (}usually called a corkscrew point of $\Delta_r${\rm )} and
\begin{equation}\label{Eqqq-26}
u_{\rm av}(X):=\dint_{B_{\delta(X)/2}(X)}u(Z)\,dZ.
\end{equation}
\end{lemma}

\begin{proof} Fix $\theta\in [1/6,6]$. We first consider the case when $r$ is small, that is $2r\le h$.
This implies that $u=u_h$ solves the PDE system $\mathcal Lu=0$ on the set we shall integrate over.

Consider the pull-back transformation $\rho:\BBR^{n}_{+}\to\mathcal{O}$ 
defined as in section \ref{SS:PT} relative to the Lipschitz function $\theta\hbar_{\nu,a}(w)$.
Let $v=(v_{\beta})_{1\leq\beta\leq N}$ be given by $v:=u\circ\rho$ in $\BBR^{n}_{+}$. 
Thanks to the assumptions made on the system \eqref{E:D}, the vector-valued function 
$v:\BBR^{n}_+\to\BBR^N$ will satisfy a PDE similar to that of $u$. Specifically, we have
\begin{equation}\label{ESv}
\left[\partial_{i}\left(\bar{A}_{ij}^{\alpha\beta}(x)\partial_{j}v_{\beta}\right)
+\bar{B}_{i}^{\alpha\beta}(x)\partial_{i}v_{\beta}\right]_{\alpha}=0,
\end{equation}
where $\bar{A}$ is uniformly elliptic and the coefficients $\bar{A}$ and $\bar{B}$ are such that
\begin{equation}\label{CarbarA}
d\overline{\mu}(x)=\left[\left(\sup_{B_{\delta(x)/2}(x)}|\nabla\bar{A}(x)|\right)^{2}
+\left(\sup_{B_{\delta(x)/2}(x)}|\bar{B}(x)|\right)^{2}\right]\delta(x) \,dx
\end{equation}
is a Carleson measure in ${\mathbb{R}}^n_{+}$. Moreover, the Carleson norm $\|\overline\mu\|_{\mathcal{C}}$ 
only depends on the Carleson norm of the original coefficients and the Lipschitz norm of the function $\hbar_{\nu,a}$. When the Lipschitz norm of this function goes to zero we have 
$$\limsup \|\overline\mu\|_{\mathcal{C}}\le \|\mu\|_{\mathcal{C}}$$
and hence the parameter $a>0$ may be chosen 
large enough so that the Lipschitz norm of the function $\theta \hbar_{\nu,a}$ is sufficiently small (at most $6/a$)
such that $\|\overline\mu\|_{\mathcal{C}}\le 2\|\mu\|_{\mathcal{C}}$.  As we have observed before for the original equation we may arrange (by change of variables) that $\bar{A}_{00}=I_{N\times N}$. This is true even if we only assume \eqref{EllipLH} as the condition implies invertibility of the matrix $\bar{A}_{00}$. Hence we can use  \eqref{hatA}-\eqref{hatB}.
\vskip1mm

Having fixed a scale $r>0$, we localize to a ball $B_r(y')$ in $\BBR^{n-1}$. 
Let $\zeta$ be a smooth cutoff function of the form $\zeta(x_0, x')=\zeta_{0}(x_0)\zeta_{1}(x')$ where
\begin{equation}\label{Eqqq-27}
\zeta_{0}= 
\begin{cases}
1 & \text{ in } [0, r], 
\\
0 & \text{ in } [2r, \infty),
\end{cases}
\qquad
\zeta_{1}= 
\begin{cases}
1 & \text{ in } B_{r}(y'), 
\\
0 & \text{ in } \mathbb{R}^{n}\setminus B_{2r}(y')
\end{cases}
\end{equation}
and
\begin{equation}\label{Eqqq-28}
r|\partial_{0}\zeta_{0}|+r|\nabla_{x'}\zeta_{1}|\leq c
\end{equation}
for some constant $c\in(0,\infty)$ independent of $r$.
Our goal is to control the $L^2$ norm of $u\big(\theta \hbar_{\nu,a}(w)(\cdot),\cdot\big)$.  
Since after the pullback under the mapping $\rho$ the latter is comparable with the $L^2$ norm 
of $v(0,\cdot)$, we fix $\alpha\in\{1,\dots,N\}$ and proceed to estimate
\begin{align}\label{u6tg}
&\hskip -0.20in
\int_{B_{2r}(y')}v_{\alpha}^{2}(0,x')\zeta(0,x')\,dx' 
\nonumber\\[4pt]
&\hskip 0.70in
=-\iint_{[0,2r]\times B_{2r}(y')}\partial_{0}\left[v_{\alpha}^{2}(x_0,x')\zeta(x_0,x')\right]\,dx_0\,dx' 
\nonumber\\[4pt]
&\hskip 0.70in
=-2\iint_{[0,2r]\times B_{2r}(y')}v_{\alpha}\partial_{0}v_{\alpha}\zeta\,dx_0\,dx'  
\nonumber\\[4pt]
&\hskip 0.70in
\quad-\iint_{[0,2r]\times B_{2r}(y')}v_{\alpha}^{2}(x_0,x')\partial_{0}\zeta\,dx_0\,dx'
\nonumber\\[4pt]
&\hskip 0.70in
=:\mathcal{A}+IV.
\end{align}
We further expand the term $\mathcal A$ as a sum of three terms obtained 
via integration by parts with respect to $x_0$ as follows:
\begin{align}\label{utAA}
\mathcal A &=-2\iint_{[0,2r]\times B_{2r}(y')}v_{\alpha}\partial_{0} 
v_{\alpha}\left(\partial_{0}x_0\right)\zeta\,dx_0\,dx' 
\nonumber\\[4pt]
&=2\iint_{[0,2r]\times B_{2r}(y')}\left|\partial_{0}v_{\alpha}\right|^{2}x_0\zeta\,dx_0\,dx' 
\nonumber\\[4pt]
&\quad +2\iint_{[0,2r]\times B_{2r}(y')}v_{\alpha}\partial^2_{00}v_{\alpha}x_0\zeta\,dx_0\,dx' 
\nonumber\\[4pt]
&\quad +2\iint_{[0,2r]\times B_{2r}(y')}v_{\alpha}\partial_{0}v_{\alpha}x_0\partial_{0}\zeta\,dx_0\,dx' 
\nonumber\\[4pt]
&=:I+II+III.
\end{align}

We start by analyzing the term $II$. In view of the fact that $\bar{A}_{00}=I_{N\times N}$, 
the PDE recorded in \eqref{ESv} allows us to write
\begin{equation}\label{S3:T8:E01}
\partial^2_{00}v_{\alpha}
=-\sum_{(i,j)\neq(0,0)}\partial_{i}\left(\bar{A}_{ij}^{\alpha\beta}\partial_{j}v_{\beta}\right)
-B_i^{\alpha\beta}\partial_iv_\beta.
\end{equation}
In turn, this permits us to express
\begin{align}\label{TFWW}
II &=-2\sum_{(i,j)\neq(0,0)}\iint_{[0,2r]\times B_{2r}}\partial_{i}\left(\bar{A}_{ij}^{\alpha\beta}\right) 
v_\alpha\partial_{j}v_{\beta}x_0\zeta\,dx_0\,dx' 
\nonumber\\[4pt]
&\quad -2\iint_{[0,2r]\times B_{2r}}B_i^{\alpha\beta}v_\alpha\partial_{i}v_{\beta}x_0\zeta\,dx_0\,dx' 
\nonumber\\[4pt]
&\quad -2\sum_{(i,j)\neq(0,0)}\iint_{[0,2r]\times B_{2r}}\bar{A}_{ij}^{\alpha\beta}v_\alpha
\partial^2_{ij}v_{\beta}x_0 \zeta\,dx_0\,dx' 
\nonumber\\[4pt]
&=:II_{1}+II_{2}+II_{3}.
\end{align}
The last term above requires some further work. Let us temporarily fix $i,j$ and denote by 
$II_{3}^{ij}$ the corresponding term in $II_3$. Since in the present context we have
$(i,j)\neq(0,0)$, at least one of the two indices involved is not zero, say $i>0$. 
Integrating by parts with respect to the variable $x_i$ then yields (in what follows we do 
not sum over indices $i$ and $j$)
\begin{align}\label{6GBBB}
II_{3}^{ij} &=2\iint_{[0,2r]\times B_{2r}}\partial_{i}\left(\bar{A}_{ij}^{\alpha\beta}\right) 
v_\alpha\partial_{j}v_{\beta}x_0\zeta\,dx_0\,dx' 
\nonumber\\[4pt]
&\quad +2\iint_{[0,2r]\times B_{2r}}\bar{A}_{ij}^{\alpha\beta}\partial_{i}v_{\alpha}\partial_jv_{\beta}   
x_0\zeta\,dx_0\,dx' 
\nonumber\\[4pt]
&\quad +2\iint_{[0,2r]\times B_{2r}}\bar{A}_{ij}^{\alpha\beta}v_\alpha\partial_{j}v_{\beta} 
x_0\partial_{i}\zeta\,dx_0\,dx' 
\nonumber\\[4pt]
&=J^{ij}_1+J^{ij}_2+J^{ij}_3.
\end{align}
The treatment of $II_{3}^{ij}$ in the case when $i=0$ proceeds along the same lines, except that 
we now integrate in the variable $x_j$. Since the resulting terms are of a similar nature as above, 
we omit writing them explicitly. 

We now group together terms that are of the same type. Firstly, we have 
\begin{equation}\label{Eqqq-29}
I+J_{2}\leq C(\lambda,\Lambda,n,N)\|S_b(u)\|^2_{L^2(B_{2r})}.
\end{equation}
Here,  the estimate would be true even with $\|S^{2r}_b(v)\|^2_{L^2(B_{2r})} $ which is at every point
dominated by $\|S_b(u)\|^2_{L^2(B_{2r})}$.
Secondly, the Carleson condition \eqref{CarbarA} and the Cauchy-Schwarz inequality imply
\begin{equation}\label{Eqqq-30}
II_{1} + II_{2}+J_1 \leq C(n,N)\|\mu\|_{\mathcal C}^{1/2} 
\|S_b(u)\|_{L^2(B_{2r})}\|\tilde{N}_a(u)\|_{L^2(B_{2r})}.
\end{equation}
Next, corresponding to the case when the derivative falls on the cutoff function $\zeta$ we have
\begin{align}\label{TDWW}
J_{3}+III &\leq C(\lambda,\Lambda,n,N)\iint_{[0,2r]\times B_{2r}}\left|\nabla v\right||v|\frac{x_0}{r}\,dx_0\,dx' 
\nonumber\\[4pt]
&\leq C(\lambda,\Lambda,n,N)\left(\iint_{[0,2r]\times B_{2r}}|v|^{2}\frac{x_0}{r^{2}}\,dx_0\,dx'\right)^{1/2} 
\|S^{2r}_b(v)\|_{L^2(B_{2r})} 
\nonumber\\[4pt]
&\leq C(\lambda,\Lambda,n,N)\|S_b(u)\|_{L^2(B_{2r})}\|\tilde{N}_a(u)\|_{L^2(B_{2r})}.
\end{align}
Finally, the interior term $IV$, which arises from the fact that $\partial_{0}\zeta$ vanishes on the set
$(0,r)\cup(2r,\infty)$ may be estimated as follows:
\begin{equation}\label{Eqqq-31}
IV\leq\frac{c}{r}\iint_{[r,2r]\times B_{2r}}|v|^{2}\,dx_0\,dx'.
\end{equation}
Summing up all terms, the above analysis ultimately yields
\begin{align}\label{E1:uonh}
&\hskip -0.20in \int_{B_{r}(y')}|v(0,x')|^2\,dx' 
\nonumber\\[4pt]
&\hskip 0.40in 
\leq C(\lambda,\Lambda,n,N)(1+\|\mu\|^{1/2}_{\mathcal C}) 
\|S_b(u)\|_{L^2(B_{2r})}\|\tilde{N}_a(u)\|_{L^2(B_{2r})}
\nonumber\\[4pt]
&\hskip 0.40in 
\quad+C(\lambda,\Lambda,n,N)\|S_b(u)\|^2_{L^2(B_{2r})}
+\frac{c}{r}\iint_{[r,2r]\times B_{2r}}|v|^2\,dx_0\,dx'.
\end{align}
With this in hand, the estimate in \eqref{TTBBMM} follows (by passing from $v$ back to $u$ via the map $\rho$).

The case $r>\!>h$ requires some extra care. Observe that for $\theta\hbar_{\nu,a}(w)(x')\ge h$ 
we have $u\big(\theta\hbar_{\nu,a}(w)(x'),x'\big)=0$, the integrand in the left-hand side of \eqref{TTBBMM} vanishes at such points. 
It follows that without loss of generality we may modify our function $\hbar_{\nu,a}$ assume that $\theta\hbar_{\nu,a}(w)\le h$ 
in $\Delta_r$ without changing the value of the left-hand side of  \eqref{TTBBMM}. What this implies is that the estimate 
\eqref{TTBBMM} for $\Delta_r$ can be deduced from adding up estimates like \eqref{TTBBMM} formulated for smaller balls 
$\Delta_{r'}\subset\Delta_r$, where $r'\approx h$. In such a scenario, we still have $\hbar_{\nu,a}\le 2r'$, 
and the desired estimate for such small balls was established above. Ultimately, we conclude that 
\eqref{TTBBMM} holds for balls of all sizes.

Finally, the last claim in the statement of the lemma can be seen as follows. 
If $\mathcal K=B_{\delta(X)/2}(X)$ and $A_r=X$ then the claim in question becomes a direct
consequence of Poicar\'e's inequality (cf. Lemma~\ref{poincare}). 
For more general $\mathcal K$, there is a finite covering of $\mathcal K$ 
by balls of the form $B_i=B_{\delta(X_i)/2}(X_i)$. Then
\begin{equation}\label{Eqqq-32}
\iint_{\mathcal{K}}|u|^2\,dX\leq\sum_i\int_{B_i}|u|^2\,dZ\leq C\sum_i r^{n-1}|u_{\rm av}(X_i)|^2
+\int_{\Delta_{2r}}S_b(u)\,d\sigma,
\end{equation}
by Poincar\'e's inequality. Furthermore, for each $i$ we have 
(abbreviating $r_i:=\delta(X_i)$, $\bar{r}:=\delta(A_r)$, and $B:=B_{1/2}(0)$):
\begin{align}\label{y5VVV}
\big|u_{\rm av}(X_i)\big|^2 &\leq 2\big|u_{\rm av}(A_r)\big|^2+2\big|u_{\rm av}(X_i)-u_{\rm av}(A_r)\big|^2
\nonumber\\[4pt]
&\leq 2\big|u_{\rm av}(A_r)\big|^2+2\left(\dint_B\big|u(X_i+r_i\xi)-u(A_r+\bar{r}\xi)\big|\,d\xi\right)^2
\nonumber\\[4pt]
&\leq 2\big|u_{\rm av}(A_r)\big|^2+2\,\dint_B\left|u(X_i+r_i\xi)-u(A_r+\bar{r}\xi)\right|^2\,d\xi.
\end{align}
Note that the last term above is of the same type as the right-hand side of \eqref{Daver}. 
As in the past, the term in question may once again be estimated as in \eqref{Daver2}. 
Hence, ultimately, this is $\leq C(S_b(u)(Q))^2$ for all $Q\in\Delta_{2r}$. The desired 
conclusion now readily follows from this.
\end{proof}

We now make use of Lemma~\ref{S3:L8}, involving the stopping time Lipschitz functions 
$\theta \hbar_{\nu,a}(w)$, in order to obtain the good-$\lambda$ inequality stated in the next lemma. As a preamble, 
we agree to let $Mf(x'):=\sup_{r>0}\dint_{|x'-z'|<r}|f(z')|\,dz'$, for $x'\in{\mathbb{R}}^{n-1}$, 
denote the standard Hardy-Littlewood maximal function on $\partial\BBR^n_+=\BBR^{n-1}$. 

\begin{lemma}\label{LGL} 
Consider the system \eqref{E:D} with coefficients satisfying the Carleson condition and \eqref{EllipLH} 
in $\BBR^n_{+}$. Then for each $\gamma\in(0,1)$ there exists a constant $C(\gamma)>0$ 
such that $C(\gamma)\to 0$ as $\gamma\to 0$ and with the property that for each $\nu>0$ and 
each energy solution $u$ of \eqref{E:D-strip} there holds 
\begin{align}\label{eq:gl}
&\hskip -0.20in 
\left|\Big\{x'\in {\BBR}^{n-1}:\,\tilde{N}_a(u)>\nu,\,(M(S^2_b(u)))^{1/2}\leq\gamma\nu,\,
\big(M(S^2_b(u))M(\tilde{N}_a^2(u))\big)^{1/4}\leq\gamma\nu\Big\}\right|
\nonumber\\[4pt] 
&\hskip 0.50in
\quad\le C(\gamma)\left|\big\{x'\in{\BBR}^{n-1}:\,\tilde{N}_a(u)(x')>\nu/32\big\}\right|.
\end{align}
\end{lemma}

\begin{proof} 
For starters, observe that $\big\{x'\in{\BBR}^{n-1}:\,\tilde{N}_a(u)(x')>\nu/32\}$ is an open 
subset of ${\BBR}^{n-1}$. When this set is empty or the entire Euclidean ambient, 
estimate \eqref{eq:gl} is trivial, so we focus on the case when the set in question is 
both nonempty and proper. Granted this, we may consider a Whitney decomposition $(\Delta_i)_{i\in I}$ 
of it, consisting of open cubes in ${\mathbb{R}}^{n-1}$. Let $F_\nu^i$ be the set appearing on the 
left-hand side of \eqref{eq:gl} intersected with $\Delta_i$. We may streamline the index set $I$ 
by retaining only those $i$'s for which $F_\nu^i\neq\varnothing$. Let $B_i$ be a ball 
of radius $r_i$ in $\BBR^n$ such that $\Delta_i\subset B_i\cap \{x_0=0\}$ and there 
exists a point $p'\in 2B_i\cap \partial\BBR^n_{+}$ with $\tilde{N}_a(u)(p')=N_a(w)(p')\leq\nu/32$. 
The existence of such point $p'$ is guaranteed by the very nature of the Whitney decomposition. 
Indeed, there exists a point near $\Delta_i$ not contained in the set 
$\{x'\in{\BBR}^{n-1}:\,\tilde{N}_a(u)(x')>\nu/32\}$.
 
This clearly implies that $w(z)\le\nu/32$ for all $z\in\Gamma_a(p')$. In particular, for all $x'\in \Delta_i$ 
we have $w(z)\le\nu/32$ for all $z\in\Gamma_a(x')\cap\Gamma_a(p')$, so we focus on estimating the size 
of $w(z)$ for $z\in\Gamma_a(x')\setminus\Gamma_a(p')$ with $z_0\geq 2r$. 
Since we also assume that for at least one $x'\in\Delta_i$ we have $M(S^2_b(u))(x')\leq(\gamma\nu)^2$, 
by the same type of estimates established in the proof of Lemma~\ref{l6} (cf. \eqref{Daver3} in particular) 
we may conclude that for sufficiently small $\gamma>0$ we have that for any $z\in\Gamma_a(x')$ with $z_0\ge 2r$ 
there is a point $\tilde{z}\in\Gamma_a(p')$ with 
\begin{equation}\label{Eqqq-33}
|z-\tilde{z}|\leq Cr_i\,\,\text{ and }\,\,|w(z)-w(\tilde{z})|\leq\nu/32.
\end{equation}
It follows that for all such $z$ we have $w(z)\le\nu/16$. Hence for all $x'\in \Delta_i$ we have
\begin{equation}\label{Eqqq-34}
\nu<\tilde{N}_a(u)(x')={N}_a(w)(x')=N_a^{2r}(w)(x'),
\end{equation}
where  $N_a^{2r}$ is the truncated nontangential maximal function at height $2r$. In particular this also implies
\begin{equation}\label{eq:hbound}
\hbar_{\nu,a}(w)\leq 2r_i\,\,\text{ pointwise on }\,\,\Delta_i.
\end{equation}

Let us also note that we can find a point $q$ (specifically, a corkscrew point for $12\Delta_i$) 
with distance to $\Delta_i$ and the boundary equal to $12r_i$ such that $w(q)\leq\nu/16$. 
When $h\lesssim r_i$ since $u$ vanishes above height $h$ and we might actually take $q$ such that $w(q)=0$.

As $w$ is the $L^2$ average of $|u|$, in terms of $u_{\rm av}(q)=\dint_{B_{\delta(q))/2}(q)}u(z)\,dz$
the latter estimate gives
\begin{equation}\label{Eqqq-35}
|u_{\rm av}(q)|\leq w(q)\le\nu/16.
\end{equation}

Next, consider $\tilde{u}:=u-u_{\rm av}(q)$. (For $h\lesssim r_i$ this is just $u$ as  $u_{\rm av}(q)=0$). Then $\mathcal{L}\tilde{u}=0$, hence $\tilde{u}$ still 
solves the system \eqref{E:D-strip} and $\tilde{u}_{\rm av}({q})=0$. Denote by $\tilde{w}$ the $L^2$ 
averages of $|\tilde u|$. For all $x'\in F^i_{\nu}$ we have
\begin{equation}\label{Eqqq-36}
N_a^{2r}(\tilde{w})(x')\geq N_a^{2r}({w})(x')-|u({q})|\geq\nu-\nu/16>\nu/2.
\end{equation}
With $\hbar:=\hbar_{\nu,a}(w)$ and for $M_\hbar$ defined on the graph of $\hbar$ 
in Corollary~\ref{S3:L6} we see that Corollary~\ref{S3:L6} 
applied to $\tilde{u}$ implies\footnote{Technically $\tilde{u}\in W^{1,2}_{\rm loc}(\Omega)$ 
is not an energy solution, but in the proof the smallness of the solution is only need above 
a certain distance from the boundary. In our case we obviously have 
$\tilde{w}(z)\leq w(z)+|u({q})|\leq\nu/8$ for points $z$ whose distance to the boundary 
exceeds $2r_i$ which suffices for our purposes.}  
\begin{equation}\label{Eqqq-37}
M_\hbar\left(\tilde{w}\chi_{4B_i}\right)\big(\hbar(x'),x'\big)\geq C(n)\nu.
\end{equation}
Here we are allowed to apply the cutoff function $\chi_{4B_i}$ since values of 
$\tilde{w}$ are small above the height $2r$, hence this places a bound on the distance 
and the diameter of the boundary ball $R$ constructed in Corollary~\ref{S3:L6} from the 
point $x'$ (both are bounded by $\lesssim r_i$). Thus, by the maximal function theorem
\begin{align}\label{UHB}
|F_\nu^i| &\leq\frac{C}{\nu^2}\int_{4\Delta_i}\big(M_\hbar(\tilde{w}\chi_{4B_i})\big)^{2}\big(\hbar(x'),x'\big)\,dx'
\nonumber\\[4pt]
&\leq\frac{C}{\nu^2}\int_{4\Delta_i}\tilde{w}^2(\hbar(x'),x')\,dx'.
\end{align}

At this stage, we bring in the following lemma. 

\begin{lemma}\label{Lw-u} 
For any surface ball $\Delta$, if $a>0$ and $\hbar=\hbar_{\nu,a}(w)$ then 
\begin{equation}\label{Eqqq-38}
\int_{\Delta}\tilde{w}^2(\hbar(x'),x')\,dx'\leq C\int_{1/6}^{6}\int_{3\Delta} 
\big|\tilde{u}(\theta \hbar(x'),x')\big|^2\,dx'\,d\theta.
\end{equation}
\end{lemma}

Accepting for the moment this lemma, whose proof we postpone for a later occasion, 
we have (taking $a>0$ as in Lemma~\ref{S3:L8})
\begin{equation}\label{eq:Fnu}
|F_\nu^i|\leq\frac{C}{\nu^2}\int_{1/6}^{6}\int_{12\Delta_i}\big|\tilde{u}(\theta \hbar(x'),x')\big|^2\,dx'\,d\theta.
\end{equation}
For each $\theta$, we apply the conclusion in Lemma~\ref{S3:L8} (in the version recorded in the very last 
part of its statement) to the solution $\tilde{u}$. This gives
\begin{align}\label{NEWEST}
&\hskip -0.20in
\int_{12\Delta_i}|\tilde{u}(\theta \hbar(x'),x')|^2\,dx' 
\nonumber\\[4pt]
&\hskip 0.20in
\leq C(1+\|\mu\|^{1/2}_{\mathcal C})\|S_{b}({u})\|_{L^2(24\Delta_i)}
\|{N}_{a}(\tilde{w})\|_{L^2(24\Delta_i)}
\nonumber\\[4pt]
&\hskip 0.20in
\quad+C\|S_{b}(u)\|^2_{L^2(24\Delta_i)}+Cr^{n-1}|\tilde{u}_{\rm av}(q)|^2
\nonumber\\[4pt]
&\hskip 0.20in
\leq C(1+\|\mu\|^{1/2}_{\mathcal C})\|S_{b}(u)\|_{L^2(24\Delta_i)}
\|{N}_{a}({w}+w({q}))\|_{L^2(24\Delta_i)}
\nonumber\\[4pt]
&\hskip 0.20in
\quad+C\|S_{b}(u)\|^2_{L^2(24\Delta_i)}.
\end{align}
Observe that we have 
dropped the term $Cr^{n-1}|\tilde{u}_{\rm av}(q)|^2$ as we have arranged previously that 
$\tilde{u}_{\rm av}(q)=0$. Since $F_\nu^i\neq\varnothing$ and $|w({q})|\leq\nu/16$ the 
term in the penultimate line of \eqref{NEWEST} may be bounded by
\begin{align}
&\hskip -0.20in
C|24\Delta_i|\left(\dint_{24\Delta_i} S_b^2(u)dx'\right)^{1/2}
\left[\left(\dint_{24\Delta_i} N_a^2(u)dx'\right)^{1/2}+\frac\nu{16}\right]
\nonumber\\[4pt]
&\hskip 0.20in
\leq C|24\Delta_i|\left[\big(M(S^2_b(u))(x')M(\tilde{N}_a^2(u))(x')\big)^{1/2} 
+\frac\nu{16}M\big(S^2_b(u)\big)(x')^{1/2}\right]
\nonumber\\[4pt]
&\hskip 0.20in
\leq C|24\Delta_i|(\gamma^2+\gamma/16)\nu^2=C(\gamma)|\Delta_i|\nu ^2.
\end{align}
Here $x'\in F_\nu^i$ is a point where we use the assumptions for the set on 
the left-hand side of \eqref{eq:gl}. Also, we have used that $|24\Delta_i|\lesssim|\Delta_i|$ by the 
doubling property of the Lebesgue measure. The estimate for the very last term of \eqref{NEWEST} is analogous. 
By design, we have $C(\gamma)\to 0$ as $\gamma\to 0$. Using this back in \eqref{eq:Fnu} we obtain
\begin{equation}\label{Eqqq-39}
|F_\nu^i|\leq C'(\gamma)|\Delta_i|.
\end{equation}
Summing over all $i$ we obtain \eqref{eq:gl}, as desired.
\end{proof}

\noindent At this stage, it remains to prove Lemma~\ref{Lw-u}.

\begin{proof}
Write $\BBR^{n-1}=\bigcup_{i\in\mathbb{Z}}\Delta^i$ where, for each $i$,  
\begin{equation}\label{Eqqq-40}
\Delta^{i}:=\big\{x'\in{\BBR}^{n-1}:\,2^{i-1}\leq \hbar(x')<2^{i}\big\}.
\end{equation}
Consider $y=(y_0,y')\in B_{\hbar(x')/2}(\hbar(x'),x')$ for $x'\in\Delta^i\cap\Delta$. Then 
\begin{equation}\label{Eqqq-41}
|y_0-h(x')|\leq \hbar(x')/2\,\,\text{ and }\,\,|x'-y'|\leq \hbar(x')/2.
\end{equation}
The goal is to estimate $\hbar(y')$. Since $h=\hbar_{\nu,a}(w)$ is a Lipschitz function 
with Lipschitz constant $1/a<1$ (cf. Lemma~\ref{S3:L5}) we have 
\begin{equation}\label{Eqqq-42}
\hbar(y')\geq \hbar(x')-|\hbar(x')-\hbar(y')|\geq \hbar(x')-|x'-y'|>\frac{\hbar(x')}{2}
\end{equation}
and
\begin{equation}\label{Eqqq-43}
\hbar(y')\leq \hbar(x')+|\hbar(x')-\hbar(y')|\leq \hbar(x')+|x'-y'|<3\hbar(x')/2.
\end{equation}
It follows that if $\mathcal{O}:=\bigcup_{x'\in\Delta^i\cap\Delta}B_{\hbar(x')/2}(\hbar(x'),x')$ then
\begin{equation}\label{Eqqq-44}
(y_0,y')\in\mathcal{O}\,\Longrightarrow\,
\left\{
\begin{array}{l}
y'\in\widetilde{\Delta^{i}}:=\Delta^{i-1}\cup\Delta^{i}\cup\Delta^{i+1},
\\[4pt]
y'\in3\Delta,
\\[4pt]
y_0\in [2^{i-2},3\cdot2^{i-1}).
\end{array}
\right.
\end{equation}
The fact that $y'\in 3\Delta$ follows from \eqref{eq:hbound}. Hence we have
\begin{align}\label{eq:estos}
&\hskip -0.20in
\int_{\Delta^{i}\cap\Delta}\tilde{w}^2(\hbar(x'),x')\,dx'
\nonumber\\[4pt]
&\hskip 0.20in
=\int_{x'\in \Delta^{i}\cap\Delta}\dint_{B_{h(x')/2}(\hbar(x'),x')}|\tilde{u}(z)|^2\,dz\,dx'
\nonumber\\[4pt]
&\hskip 0.20in
\leq C2^{-in}\int_{x'\in\Delta^{i}\cap\Delta}\int_{B_{\hbar(x')/2}(h(x'),x')}|\tilde{u}(z)|^2\,dz\,dx'
\nonumber\\[4pt]
&\hskip 0.20in
=C2^{-in}\iint_{\mathcal O}|\tilde{u}(z)|^2\left|\{x'\in\Delta^i\cap\Delta:\,z\in B_{\hbar(x')/2}(\hbar(x'),x')\}\right|\,dz,
\end{align}
where in the last step we have interchanged the order of integration. For a fixed $z\in\mathcal{O}$ we have
\begin{align}\label{Eqqq-45}
&\hskip -0.20in
\left|\{x'\in\Delta^i\cap\Delta:\,z\in B_{\hbar(x')/2}(\hbar(x'),x')\}\right|
\nonumber\\[4pt]
&\hskip 0.50in
\leq\left|\{x'\in{\BBR}^{n-1}:\,z\in B_{\hbar(x')/2}(\hbar(x'),x')\}\right|.
\end{align}
Since for such $z=(z_0,z')$ we have $z_0\in [2^{i-2},3\cdot 2^{i-1})$ and 
\begin{equation}\label{Eqqq-46}
\frac{\hbar(x')}{2}<z_0<\frac{3\hbar(x')}{2}\,\Longrightarrow\,\hbar(x')\in(2z_0/3,2z_0)\subset(2^{i-1}/3,3\cdot 2^{i}).
\end{equation}
From this we then conclude 
\begin{equation}\label{Eqqq-47}
\{x'\in{\BBR}^{n-1}:\,z\in B_{h(x')/2}(h(x'),x')\}\subset\{x'\in{\BBR}^{n-1}:\,|x'-z'|<2^{i+2}\}
\end{equation}
hence, further, 
\begin{equation}\label{Eqqq-48}
\left|\{x'\in\Delta^i\cap\Delta:\,z\in B_{\hbar(x')/2}(\hbar(x'),x')\}\right|\leq C2^{i(n-1)}.
\end{equation}
Using this back in \eqref{eq:estos} then yields
\begin{align}\label{eq:estos2}
\int_{\Delta^{i}\cap\Delta}\tilde{w}^2(\hbar(x'),x')\,dx' &\leq C2^{-i}\iint_{\mathcal{O}}|\tilde{u}(z)|^2\,dz
\nonumber\\[4pt]
&\leq C\int_{z'\in\mathcal{P}(\mathcal{O})}\dint_{z_0\in(2^{i-2},3\cdot 2^{i-1})}|\tilde{u}(z)|^2\,dz_0\,dz',
\end{align}
where (with $\widetilde{\Delta^{i}}$ as in \eqref{Eqqq-44})  
\begin{equation}\label{Eqqq-49}
\mathcal{P}(\mathcal{O}):=\{z'\in{\BBR}^{n-1}:\,\exists\,z_0\,\text{ such that }\,(z_0,z')\in\mathcal{O}\}
\subset\widetilde{\Delta^{i}}\cap 3\Delta.
\end{equation}
Clearly since for $z'\in \mathcal{P}(\mathcal{O})$ we have $\hbar(z')\in [2^{i-2},3\cdot 2^{i-1})$ and, 
therefore,
\begin{equation}\label{Eqqq-50}
(2^{i-2}, 3\cdot 2^{i-1})\subset(\hbar(z')/6,6\hbar(z')).
\end{equation}
Hence \eqref{eq:estos2} may be also written as
\begin{equation}\label{eq:estos3}
\int_{\Delta^{i}\cap \Delta} \tilde{w}^2(\hbar(x'),x')\,dx'
\leq C\int_{\widetilde{\Delta^{i}}\cap 3\Delta}\int_{1/6}^6|\tilde{u}(\theta \hbar(z'),z')|^2\,d\theta\,dz'.
\end{equation}
By interchanging the order of integration and then summing over all $i\in\mathbb Z$ we arrive at 
\begin{align}\label{eq:estos4}
&\hskip -0.20in
\int_{\Delta}\tilde{w}^2(h(x'),x')\,dx'
\nonumber\\[4pt]
&\hskip 0.20in
\leq C\int_{1/6}^6\sum_i\int_{\widetilde{\Delta^{i}}\cap 3\Delta} 
|\tilde{u}(\theta \hbar(z'),z')|^2\,dz'\,d\theta
\nonumber\\[4pt]
&\hskip 0.20in
=C\int_{1/6}^6\sum_i\left(\int_{\Delta^{i-1}\cap 3\Delta}
+\int_{\Delta^{i}\cap 3\Delta}+\int_{\Delta^{i+1}\cap 3\Delta}\right) 
|\tilde{u}(\theta \hbar(z'),z')|^2\,dz'\,d\theta
\nonumber\\[4pt]
&\hskip 0.20in
=3C\int_{1/6}^6\sum_i\int_{\Delta^{i}\cap 3\Delta}|\tilde{u}(\theta \hbar(z'),z')|^2\,dz'\,d\theta
\nonumber\\[4pt]
&\hskip 0.20in
=3C\int_{1/6}^6\int_{3\Delta}|\tilde{u}(\theta \hbar(z'),z')|^2\,dz'\,d\theta,
\end{align}
as wanted. This finishes the proof of Lemma~\ref{Lw-u} and completes the proof of
Lemma~\ref{LGL}.  
\end{proof}

Lemma \ref{LGL} has a localized version on any boundary ball $\Delta_d\subset {\mathbb R}^{n-1}$. 

\begin{lemma}\label{LGL-loc} 
Consider the system \eqref{E:D} with coefficients satisfying the Carleson condition and \eqref{EllipLH} 
in $\BBR^n_{+}$.  Consider any boundary ball $\Delta_d=\Delta_d(Q)\subset {\mathbb R}^{n-1}$, 
let $A_d=(d/2,Q)$ be its corkscrew point, and let
\begin{equation}\label{TX!!!}
\nu_0=\left(\dint_{B_{d/4}(A_d)}|u(z)|^2\,dz\right)^{1/2}.
\end{equation}

Then for each $\gamma\in(0,1)$ there exists a constant $C(\gamma)>0$ 
such that $C(\gamma)\to 0$ as $\gamma\to 0$ and with the property that for each $\nu>2\nu_0$ and 
each energy solution $u$ of \eqref{E:D} there holds 
\begin{align}\label{eq:gl2}
&\hskip -0.20in 
\Big|\Big\{x'\in {\BBR}^{n-1}:\,\tilde{N}_a(u\chi_{T(\Delta_d)})>\nu,\,(M(S^2_b(u)))^{1/2}\leq\gamma\nu,
\nonumber\\[4pt] 
&\hskip 0in
\big(M(S^2_b(u))M(\tilde{N}_a^2(u\chi_{T(\Delta_d)}))\big)^{1/4}\leq\gamma\nu\Big\}\Big|
\nonumber\\[4pt] 
&\hskip 0.50in
\quad\le C(\gamma)\left|\big\{x'\in{\BBR}^{n-1}:\,\tilde{N}_a(u\chi_{T(\Delta_d)})(x')>\nu/32\big\}\right|.
\end{align}
Here $\chi_{T(\Delta_d)}$ is the indicator function of the Carleson region $T(\Delta_d)$ and the square function
$S_b$ in \eqref{eq:gl2} is truncated at the height $2d$. Similarly, the Hardy-Littlewood maximal operator $M$
is only considered over all balls $\Delta'\subset\Delta_{md}$ for some enlargement constant $m=m(a)\ge 2$.
\end{lemma}

\begin{proof}
The proof is similar to Lemma \ref{LGL}, hence we only point out the main differences introduced by 
considering $\tilde{N}$ of $u\chi_{T(\Delta_d)}$ instead of $u$. Let $w_1$ be the $L^2$ averages of 
$u\chi_{T(\Delta_d)}$ instead of $u$. If we consider $\hbar_{\nu,a}(w_1)$ as in \eqref{h}, then Lemma~\ref{S3:L5} 
holds for $\hbar_{\nu,a}(w_1)$ as before. 

Because $u\chi_{T(\Delta_d)}=0$ outside $T(\Delta_d)$, it follows that $w_1$ vanishes outside $T(\Delta_{2d})$.
As such, it follows that there exists a number $m=m(a)>1$ with the property that
\begin{equation}\label{h-est}
N_a(w_1)=0\mbox{ on }{\mathbb R}^{n-1}\setminus\Delta_{md}\quad\mbox{and}\quad 
\sup_{{\mathbb R}^{n-1}}\hbar_{\nu,a}(w_1)=\sup_{\Delta_d}\hbar_{\nu,a}(w_1).
\end{equation}

We only consider $\nu>2\nu_0$. We claim that for such choice of $\nu$, Lemma~\ref{l6}, hence also Corollary \ref{S3:L6}, 
remain valid and only require minor changes, which we outline below. Start by choosing $b=b(a)>0$ such that whenever 
$x'\in\Delta_{md}$ we have
$$
\left[{d}/{48},d\right]\times \Delta_d \subset \Gamma_b(x').
$$
Granted this, \eqref{Daver3} implies 
$$
|w(A_d)-w(y_0,y')|^2 \lesssim C\gamma^2\nu^2
\,\,\text{ for all }\,\,(y_0,y')\in\left[{d}/{24},2d\right]\times\Delta_d.
$$
Here, as before, $w$ denotes the $L^2$ averages of un-truncated function $u$. 
Choose $\gamma\le\gamma_0$ where $C\gamma_0^2=1/4$. Since $w(A_d)=w_1(A_d)$, we therefore obtain a one-sided estimate
$$
|w_1(y_0,y')|\le|w(y_0,y')|\le|w(A_d)|+|w(y_0,y')-w(A_d)|<\nu/2+\nu/2=\nu.
$$
In particular, this implies that on $\Delta_d$ we have $\hbar_{\nu,a}(w_1)\le d/24$ hence, 
thanks to \eqref{h-est}, it follows that $\hbar_{\nu,a}(w_1)\le d/24$ everywhere.

With this at our disposal, the proof of Lemma~\ref{l6} only requires one other minor modification. 
Again, find a point $y=(y_0,y')\in\partial\Gamma_a(x_0,x')$ such that $w_1(y)=\nu$, and define $R$ as before.
Consider the sub-region $R'$ of $R$ defined as 
$$
R'=\big\{z'\in R:\,|B_{z_0/2}(z_0,z)\cap T(\Delta_d)|
\ge |B_{y_0/2}(y_0,y)\cap T(\Delta_d)|/2\mbox{ for all }(z_0,z')\in\mathcal O\big\}.
$$
Simple geometric considerations dictate that $R\subset 4R'$. Now repeating the calculation \eqref{eqFC}
for any pair of points $z_1\in B_{z_0/2}(z_0,z')\cap T(\Delta_d)$ and $z_2\in B_{y_0/2}(y_0,y')\cap T(\Delta_d)$
we obtain a bound from below on the size of $w_1(z_0,z)$ it terms of $w_1(y_0,y')$ (it is a calculation similar to 
\eqref{Daver2} but slightly trickier, as the sets $B_{z_0/2}(z_0,z)\cap T(\Delta_d)$ and $B_{y_0/2}(y_0,y)\cap T(\Delta_d)$ 
are not necessary balls any more). We obtain
$$
w_1(z_0,z)\ge w_1(y_0,y')/2-C(a,n,N)\gamma\nu>\nu/2-\nu/4=\nu/4,
$$
for $\gamma$ chosen such that $C(a,n,N)\gamma<1/4$. It follows that Lemma~\ref{l6} 
holds for $w_1$ with $R'$ replacing $R$ and with slightly weaker conclusions, namely 
$x'\in 8R'$ and 
$$
|w_1(\hbar_{\nu,a}(w_1)(z'),z')|>\nu/4\,\,\,\mbox{ for all }\,\,z'\in R',
$$
in place of \eqref{Eqqq-18}. However, this is still sufficient to conclude that 
Corollary~\ref{S3:L6} holds for $w_1$ as well. 

We now look at Lemma~\ref{S3:L8} and, in particular, the place it is actually employed in the proof of the 
good-$\lambda$ inequality in Lemma~\ref{LGL}. Recall that we apply this lemma in one place only, namely the 
estimate in $\eqref{eq:Fnu}$, where $\Delta_i$ are Whitney cubes. Hence, we might as well arrange that the 
balls $\Delta_r$ we consider in Lemma \ref{S3:L8} are from a dyadic grid in ${\mathbb R}^{n-1}$. Similarly, 
in the claim of Lemma~\ref{LGL-loc} it suffices to consider $\Delta_d$ dyadic. 

Hence, whenever $\Delta_r\cap \Delta_d\ne\varnothing$ then either $\Delta_r\subset\Delta_d$, 
or $\Delta_d\subset\Delta_r$. If $\Delta_d\subset\Delta_r$ then clearly if we prove the claim of Lemma~\ref{S3:L8}
for $\Delta_d$ and function $u\chi_{T(\Delta_d)}$ then this will also hold for the larger ball $\Delta_r$ as the 
left-hand side of \eqref{TTBBMM} vanishes outside $\Delta_d$. The terms in the right-hand side will be bigger 
or comparable if we replace $\Delta_d$ by $\Delta_r$ there. This is also true for the last term in \eqref{TTBBMM}
because although we have $\Delta_d\subset\Delta_r$ we must have $r\approx d$. This is due to the fact that
$\Delta_r$ comes from Whitney decomposition of the set $\{\tilde{N}(w_1)>\nu/32\}\subset \Delta_{md}$, 
implying the inequality $r\lesssim d$.

Hence, it suffices to consider $\Delta_r\subset\Delta_d$, or $\Delta_r\cap\Delta_d=\varnothing$ 
in Lemma~\ref{S3:L8}. We consider these two cases separately below. 

\begin{itemize}
\item Assume $\Delta_r\cap\Delta_d=\varnothing$. Then Lemma~\ref{S3:L8} hold trivially for $u\chi_{T(\Delta_d)}$,
as the function vanishes on $\Delta_r$.

\item Assume $\Delta_r\subset\Delta_d$. We have already established above that $\hbar_{a,\nu}\le d/24$, 
therefore $\theta \hbar_{a,\nu}\le d/4$. It follows that all terms in \eqref{TTBBMM} are either the 
same, or comparable, when $u$ is replaced by $u\chi_{T(\Delta_d)}$ in the left-hand side of \eqref{TTBBMM} and 
in the term $\tilde{N}_a$, as the functions $u$ and $u\chi_{T(\Delta_d)}$ coincide in $\Delta_r\times (0,d)$. 
Also, clearly, the estimate in \eqref{TTBBMM} only requires truncated versions of $S_b$ and $\tilde{N}$.
\end{itemize}

Therefore, we may employ Lemma~\ref{S3:L8} to prove Lemma~\ref{LGL-loc} the same way as we did in the case of 
Lemma~\ref{LGL}. This shows that the local good-$\lambda$ inequality \eqref{eq:gl2} holds.
\end{proof}

Finally we have the following.

\begin{proposition}\label{S3:C7} 
Assume the coefficients satisfy the Legendre-Hadamard condition \eqref{EllipLH}
and suppose the measure $\mu$ defined as in \eqref{Car_hatAA} is Carleson in $\Omega=\BBR^n_{+}$. 
The for any $p>0$ and $a>0$ there exists an integer $m=m(a)\ge 2$ and a finite constant 
$C=C(n,N,\lambda,\Lambda,p,a,\|\mu\|_{\mathcal C})>0$ such that for any energy solution $u$ 
of \eqref{E:D-strip} in $\Omega$ and any surface ball $\Delta_d\subset{\mathbb R}^{n-1}$ we have
\begin{equation}\label{S3:C7:E00ooloc}
\|\tilde{N}^r_a(u)\|_{L^{p}(\Delta_d)}\le C\|S^{2r}_a(u)\|_{L^{p}(\Delta_{md})}+Cd^{(n-1)/p}|u_{av}(A_d)|,
\end{equation}
where $A_d$ denotes the corkscrew point of the ball $\Delta_d$, and $u_{av}$ is as in \eqref{Eqqq-26}.

Moreover, a global estimate is also valid. Specifically, for any $p>0$ and $a>0$ there exists a finite 
constant $C=C(n,N,\lambda,\Lambda,p,a,\|\mu\|_{\mathcal C})>0$ such that 
\begin{equation}\label{S3:C7:E00oo}
\|\tilde{N}_a(u)\|_{L^{p}({\BBR}^{n-1})}\le C\|S_a(u)\|_{L^{p}({\BBR}^{n-1})}.
\end{equation}
\end{proposition}

\begin{proof} When $p>2$, the local estimate claimed in \eqref{S3:C7:E00ooloc} follows immediately after 
multiplying the good-$\lambda$ inequality \eqref{eq:gl2} by $\nu^{p-1}$ and integrating in $\nu$ 
over the interval $(2\nu_0,\infty)$. Note that the fact that the square function $S^{2r}_a$ is only 
integrated over some enlargement of $\Delta_d$ instead of the whole ${\mathbb R}^{n-1}$ follows from 
the fact that the set $\{x'\in{\BBR}^{n-1}:\,\tilde{N}_a(u\chi_{T(\Delta_d)})(x')>\nu/32\big\}$ in the 
right-hand side of \eqref{eq:gl2} is contained in a ball of diameter comparable to $\Delta_d$. 
For this reason, the maximal operators $M$ in \eqref{eq:gl2} may be restricted to such an enlarged ball $\Delta_{md}$. 

We do not quite obtain \eqref{S3:C7:E00ooloc}, as in the right-hand side we get 
\begin{equation}\label{S3:C7:E00ooloc2}
\|S^{2r}_a(u)\|_{L^{p}(\Delta_{md})}+d^{(n-1)/p}\left(\dint_{B_{d/4}(A_d)}|u(z)|^2dz\right)^{1/2},
\end{equation}
but then using Poincar\'e's inequality as in \eqref{Eqqq-32}-\eqref{y5VVV}, the second term above may be estimated as 
$$
d^{(n-1)/p}\left(\dint_{B_{d/4}(A_d)}|u(z)|^2dz\right)^{1/2} \lesssim  d^{(n-1)/p}|u_{av}(A_d)|
+\left(\int_{\Delta_{2d}}\left[S^{2r}_b(u)(Q)\right]^2\,dQ\right)^{1/2}.
$$
The argument proving \eqref{S3:C7:E00oo} for all $p>0$ may be found in \cite{FSt}.  
The local estimate \eqref{S3:C7:E00ooloc} for $p>2$ is the necessary ingredient for what 
is otherwise a purely abstract, real-variable argument. Further details can be found in \cite{FSt}.
\end{proof}

\section{$L^p$ Dirichlet problem for $p$ near $2$.}
\label{S-TX}

Following \cite{DK} we explore the extrapolation of solvability from $L^2$ to $L^p$ values of $p$ near $2$.
Consider first the extrapolation to values $p>2$. In this case we invoke Theorem 1.2 of \cite{S3} which 
establishes solvability of the $L^p$ Dirichlet problem for all $2<p<\frac{2(n-1)}{n-2}+\varepsilon$ 
for some small $\varepsilon>0$, provided $L^2$ Dirichlet problem is solvable and the boundary 
Cacciopoli inequality (cf. Proposition~\ref{caccioB}) holds. As we have already established both, 
the desired conclusion follows. 

\vglue1mm

We now turn to the case $2-\varepsilon<p<2$. Following the real variable argument of \cite{DKV2}
we work with two family of cones $\Gamma_b(\cdot)$ and $\Gamma_a(\cdot)$ with $b<a$ so that the cones 
$\Gamma_a(Q)$ contain $\overline{\Gamma_b(Q)}\setminus Q$. For ease of notation, introduce 
$$
m(x')=(\tilde{N}_b u)(x'),\qquad \overline{m}(x')=(\tilde{N}_a u)(x'),
$$
and for each $\nu>0$ define 
$$
F_{\nu}=\{x'\in{\mathbb R}^{n-1}:\,\overline{m}(x')\le\nu\}.
$$
Finally, let 
$$
\widetilde{F_\nu}=\bigcup_{Q\in F_\nu}\Gamma_a(Q).
$$
Clearly, the $L^2$ solvability result from Theorem \ref{S3:T1} applies to domain $\widetilde{F_\nu}$ 
as this is a domain with Lipschitz constant $1/a$. Since ${\mathcal L}$ satisfies the assumptions 
of this theorem in ${\mathbb R}^{n}_+$, it also satisfies similar assumptions in the domain $\widetilde{F_\nu}$, provided 
$1/a$ is sufficiently small. We fix $a>0$ for which we have such solvability. Theorem \ref{S3:T1} then implies the estimate
\begin{equation}\label{Main-EstYY}
\|\tilde{N}_{a/2} u\|_{L^{2}(\partial \widetilde{F_\nu})}
\leq C\|u\big|_{\partial \widetilde{F_\nu}}\|_{L^{2}(\partial \widetilde{F_\nu})},
\end{equation}
for all energy solutions $u$ of $\mathcal Lu=0$. The constant $C>0$ in the estimate above only 
depends on $a$. Here the nontangential maximal function $\tilde N$ must be taken with respect 
to nontangential approach regions that are contained inside ${\mathcal O}_{\Delta_d,a}$, that is, 
we need to take regions $\Gamma_b(\cdot)$ for any $b<a$. Without loss of generality, choose 
$b=a/2$ and fix it for the remaining portion of this section.

Based on \eqref{Main-EstYY} we conclude that
\begin{equation}
\int_{F_\nu}m^2(x')\,dx'\le C\int_{F_\nu}|f|^2dx'+C\int_{\partial \widetilde{F_\nu}\setminus F_\nu}u^2\,d\sigma
\end{equation}
where the second term can be estimated by $C\nu^2\sigma({\mathbb R}^{n-1}\setminus F_\nu)$ by averaging 
(we vary $a$ slightly to create a solid integral out of the last term) and using the definition of the set 
$\widetilde{F_\nu}$. Hence we have
\begin{equation}\label{P2-eq1}
\int_{F_\nu}m^2(x')\,dx'\le C\int_{F_\nu}f^2dx'+C\nu^2\sigma({\mathbb R}^{n-1}\setminus F_\nu).
\end{equation}
Now, as in \cite{DKV2}, we have
\begin{equation}\label{P2-eq2}
\int_{{\mathbb R}^{n-1}}m^{2-\varepsilon}(x')\,dx'
\le C\int_{{\mathbb R}^{n-1}}m^{2}(x')\overline{m}^{-\varepsilon}(x')\,dx',
\end{equation}
which follows from the fact that for any $\varepsilon\in(0,1)$ the function $M(m)^{\varepsilon}$ 
is a Muckenhoupt weight of class $A_1$. In particular, the said function is an $A_2$ weight, hence 
so is $M(m)^{-\varepsilon}$. Consequently, Muckenhoupt's theorem guarantees that the maximal operator 
is bounded on $L^2({\mathbb R}^{n-1},\,M(m)^{-\varepsilon}dx')$, hence
\begin{align}\label{TX.2}
\int_{{\mathbb R}^{n-1}}M(m)^2(x')M(m)^{-\varepsilon}(x')\,dx'
&\le\int_{{\mathbb R}^{n-1}}m^2(x')M(m)^{-\varepsilon}(x')\,dx'
\nonumber\\[6pt]
&\le\int_{{\mathbb R}^{n-1}}m^2(x')\overline{m}^{-\varepsilon}(x')\,dx',
\end{align}
where the last estimates uses the pointwise bound $\overline{m}(x')\le CM(m)(x')$.

In turn, \eqref{P2-eq2} implies that
\begin{align}\label{TX.3}
\int_{{\mathbb R}^{n-1}}m^{2-\varepsilon}(x')\,dx'
&\le C\int_{{\mathbb R}^{n-1}}m^{2}(x')\overline{m}^{-\varepsilon}(x')\,dx'
\nonumber\\[6pt]
&=\varepsilon\int_0^\infty\nu^{-1-\varepsilon}\left(\int_{\{x':\overline{m}(x')\le\nu\}}m^2(y')dy'\right)d\nu.
\end{align}
In concert with \eqref{P2-eq1} this further permits us to estimate
\begin{align}\label{TX.4}
\int_{{\mathbb R}^{n-1}}m^{2-\varepsilon}(x')\,dx'
&\le C\varepsilon\int_0^\infty\nu^{-1-\varepsilon}
\left(\int_{\{x':\overline{m}(x')\le\nu\}}|f|^2(y')dy'\right)d\nu
\nonumber\\[6pt]
&\quad+C\varepsilon\int_0^\infty\nu^{1-\varepsilon}\sigma(\{x':\overline{m}(x')>\nu\})d\nu
\nonumber\\[6pt]
&\le C\int_{{\mathbb R}^{n-1}}|f|^2\overline{m}^{\,\,-\varepsilon}dx'
+\varepsilon\int_{{\mathbb R}^{n-1}}\overline{m}^{\,2-\varepsilon}\,dx'.
\end{align}
By classical arguments (cf., e.g., \cite{FSt}) we have
$$
\int_{{\mathbb R}^{n-1}}\overline{m}^{\,2-\varepsilon}\,dx'\lesssim \int_{{\mathbb R}^{n-1}}{m}^{2-\varepsilon}\,dx',
$$
and for some sufficiently small $\varepsilon>0$ this yields
\begin{equation}\label{TX.5}
\int_{{\mathbb R}^{n-1}}m^{\,2-\varepsilon}(x')\,dx'\le C\int_{{\mathbb R}^{n-1}}|f|^2\overline{m}^{\,\,-\varepsilon}dx'.
\end{equation}
Since for almost every $x'$ we have $|f(x')|\le \overline{m}(x')$, this ultimately yields the desired estimate
\begin{equation}\label{TX.6}
\int_{{\mathbb R}^{n-1}}m^{2-\varepsilon}(x')\,dx'\le C\int_{{\mathbb R}^{n-1}}|f(x')|^{2-\varepsilon}dx',
\end{equation}
proving solvability for $p<2$ close to $2$.

\section{Proof of Corollary \ref{MM:C1}}

In this section we show how Theorem \ref{S3:T1}  can be applied to the Lam\'e system \eqref{Lame}.
Recall that in order to apply this result we have to modify the coefficients of our system so that $A_{0j}^{\alpha\beta}=\delta_{\alpha\beta}\delta_{0j}$ and verify that the system is strongly elliptic. Also, thanks to the observation we have made earlier in \eqref{eqLM} there is a degree of flexibility in the choice of an arbitrary function $r\in L^\infty(\Omega)$. This will be useful when verifying the ellipticity condition.

Recall subsection \ref{SS:Nor} of this paper where we have explained the process of rewriting our system in a more convenient form where $A_{0j}^{\alpha\beta}=\delta_{\alpha\beta}\delta_{0j}$ holds. Following the notation we have introduced in 
subsection \ref{SS:Nor} the minor matrices $A_{ij}$ of size $n\times n$ for the coefficients given by \eqref{eqLM} are:

\begin{equation}
A_{ii} = 
\begin{blockarray}{*{5}{c} l}
&  & {i} &  &  \\
\begin{block}{[*{5}{c}] l}
    \mu & 0 & \dots & \dots & 0 & \\
    \vdots & \ddots & & & \vdots &\\
    \vdots &\dots & \lambda+2\mu &\dots&\vdots & i\\
    \vdots &&&\ddots& 0&\\
    0 &   \dots     & \dots& 0&\mu&\\
    \end{block}
 \end{blockarray}\qquad
 A_{ij}=\begin{blockarray}{*{5}{c} l}
&   {i}& j&  &  \\
\begin{block}{[*{5}{c}] l}
    0 & 0 & \dots & \dots & 0 & \\
    \vdots & \ddots &\lambda+\mu-\gamma & & \vdots & i\\
    \vdots &\gamma & 0 &\dots&\vdots &j\\
    \vdots &&&\ddots& 0&\\
    0 &   \dots     & \dots& 0& 0&\\
    \end{block}
 \end{blockarray}.\nonumber
\end{equation}
Here, $i\ne j$ and we have dropped dependence of the matrices on $x\in\Omega$ and set $\gamma=\mu-r$. In particular, $A_{00}$ is a diagonal matrix that is invertible if $\mu,\lambda+2\mu\ne 0$. We shall assume that $\mu,\lambda+2\mu> 0$ as we look for positive definitness. Assuming this we get for $\overline{A}$ defined by \eqref{hatA} and \eqref{yr4DDF}:

\begin{eqnarray}\label{eq-Abar}
&&\overline{A_{00}} =I_{n\times n},\qquad \overline{A_{0j}} =0_{n\times n},\qquad 
\overline{A_{i0}}=\begin{blockarray}{*{5}{c} l}
0& &  {i}&   &  \\
\begin{block}{[*{5}{c}] l}
    0 & \dots & \frac{\lambda+\mu}{\lambda+2\mu} & \dots & 0 & 0\\
    \vdots & \ddots && & \vdots & \\
    \frac{\lambda+\mu}\mu &0 & 0 &\dots&\vdots &i\\
    \vdots &&&\ddots& 0&\\
    0 &   \dots     & \dots& 0& 0&\\
    \end{block}
 \end{blockarray},\\
&&\overline{A_{ii}} = 
\begin{blockarray}{*{5}{c} l}
&  & {i} &  &  \\
\begin{block}{[*{5}{c}] l}
    \frac\mu{\lambda+2\mu} & 0 & \dots & \dots & 0 & \\
    \vdots & 1 & & & \vdots &\\
    \vdots &\dots & \frac{\lambda+2\mu}{\mu} &\dots&\vdots & i\\
    \vdots &&&\ddots& 0&\\
    0 &   \dots     & \dots& 0&1&\\
    \end{block}
 \end{blockarray}\qquad
\overline{ A_{ij}}=\begin{blockarray}{*{5}{c} l}
&   {i}& j&  &  \\
\begin{block}{[*{5}{c}] l}
    0 & 0 & \dots & \dots & 0 & \\
    \vdots & \ddots &\frac{\lambda+\mu-\gamma}{\mu} & & \vdots & i\\
    \vdots &\frac{\gamma}\mu & 0 &\dots& \vdots &j\\
    \vdots &&&\ddots& 0&\\
    0 &   \dots     & \dots& 0& 0&\\
    \end{block}
 \end{blockarray}.\nonumber
\end{eqnarray}
Here $i,j>0$ and $i\ne j$. Hence $\overline{A_{0j}^{\alpha\beta}}=\delta_{\alpha\beta}\delta_{0j}$ as desired. It remains to pick optimal $\gamma$ in order for the tensor $\overline{A}$ to be strongly elliptic.
For $\overline{A_{ij}^{\alpha\beta}}\eta_i^\alpha\eta_j^\beta$ we therefore have:

\begin{eqnarray}\nonumber
 \overline{A_{ij}^{\alpha\beta}}\eta_i^\alpha\eta_j^\beta&=& |\eta_0^0|^2+\textstyle\frac{\lambda+2\mu}{\mu}\sum_{i>0}|\eta_i^i|^2+\frac{\lambda+\mu}{\mu}\sum_{i>0}\eta_0^0\eta_i^i+2\frac{\lambda+\mu-\gamma}{\mu} \sum_{0<i<j}\eta_i^i\eta_j^j\\\label{eq-45}
 &+&\textstyle\sum_{i>0}\left(|\eta_i^0|^2+\frac\mu{\lambda+2\mu}|\eta_0^i|^2+\frac{\lambda+\mu}{\lambda+2\mu}\eta^0_i\eta^i_0 \right)\\\nonumber
 &+&\textstyle\sum_{0<i<j}\left(|\eta_i^j|^2+|\eta_j^i|^2+\frac{2\gamma}{\mu}\eta_i^j\eta_j^i
 \right).
\end{eqnarray}
In order for \eqref{eq-45} to be positive definite, all three lines must be positive definite. For the last line of \eqref{eq-45} that implies that $|\gamma|<\mu$. The second line of \eqref{eq-45} implies that
\begin{equation}\label{eqss}
\frac\mu{\lambda+2\mu}>\left(\frac{\lambda+\mu}{2\lambda+4\mu}\right)^2\quad\Leftrightarrow\quad
7\mu^2+2\mu\lambda-\lambda^2>0,
\end{equation}
or $1-\sqrt{8}<\frac\lambda\mu<1+\sqrt{8}$. Recall that we need $\lambda+2\mu>0$, but our new condition is stronger. Thus we shall assume that $(1-\sqrt{8})\mu<\lambda<(1+\sqrt{8})\mu$. Finally, in order for the first line of \eqref{eq-45} to be positive definite the matrix 
$$
M=\begin{bmatrix}
    1 & \frac{\lambda+\mu}{2\mu} & \dots & \dots & \frac{\lambda+\mu}{2\mu} \\
    \frac{\lambda+\mu}{2\mu} & \frac{\lambda+2\mu}{\mu} &\frac{\lambda+\mu-\gamma}{\mu} &\dots & \frac{\lambda+\mu-\gamma}{\mu} \\
    \vdots &\frac{\lambda+\mu-\gamma}{\mu}& \ddots &&\vdots \\
    \vdots &\vdots&&\ddots&\frac{\lambda+\mu-\gamma}{\mu}\\
    \frac{\lambda+\mu}{2\mu}&   \frac{\lambda+\mu-\gamma}{\mu}     & \dots& \frac{\lambda+\mu-\gamma}{\mu}& \frac{\lambda+2\mu}{\mu}\\
    \end{bmatrix}$$
must be positive definite. The lower-right $(n-1)\times(n-1)$ block has a special structure. For $\xi'=(0,\xi_1,\xi_2,\dots,\xi_{n-1})^T$ we see that 
$$\xi'^TM\xi'=\textstyle\frac{\mu+\gamma}{\mu}(|\xi_1|^2+|\xi_2|^2+\dots+|\xi_{n-1}|^2)+\textstyle\frac{\lambda+\mu-\gamma}{\mu}(\xi_1+\xi_2+\dots+\xi_{n-1})^2.$$
Hence if $\lambda+\mu-\gamma\ge 0$ this is positive definite (recall that $|\gamma|<\mu$). Using this we can write for $\xi=(\xi_0,\xi_1,\xi_2,\dots,\xi_{n-1})^T$
\begin{eqnarray}
\xi^TM\xi&=&\textstyle\frac{\mu+\gamma}{\mu}(|\xi_1|^2+|\xi_2|^2+\dots+|\xi_{n-1}|^2)+\\
&&\nonumber\textstyle\frac{\lambda+\mu-\gamma}{\mu}T^2+\frac{\lambda+\mu}{\mu}\xi_0T+|\xi_0|^2.
\end{eqnarray}
Here $T=\sum_{i>0}\xi_i$. The expression in the last line is positive definite if
\begin{equation}\label{eq-46}
\frac{\lambda+\mu-\gamma}{\mu}>\left(\frac{\lambda+\mu}{2\mu}\right)^2.
\end{equation}
It follows that we want to chose $\gamma$ as small as possible it order to make the lefthand side as large as possible. Given the restriction $|\gamma|<\mu$ the optimal choice is $\gamma=-\mu+\varepsilon$ for some small $\varepsilon>0$. Using this choice of $\gamma$ \eqref{eq-46} will hold if
\begin{equation}\label{eq-47}
\frac{\lambda+2\mu}{\mu}>\left(\frac{\lambda+\mu}{2\mu}\right)^2\quad\Leftrightarrow\quad
7\mu^2+2\mu\lambda-\lambda^2>0.
\end{equation}
Compare this to \eqref{eqss}! Thus as before we must have $(1-\sqrt{8})\mu<\lambda<(1+\sqrt{8})\mu$.

We summarise our calculations in the following lemma.

\begin{lemma}\label{Lamma-ell} The Lam\'e system 
\begin{equation}\label{Lamex}
\mathcal Lu=\nabla\cdot\left(\lambda(x)(\nabla\cdot u)I+\mu(x)(\nabla u+(\nabla u)^T) \right)=0.
\end{equation}
for an unknown function $u:\Omega\to{\mathbb R}^n$ can be written in an equivalent form as
\begin{equation}\label{Lamexx}{\mathcal L}'u=0,\quad\mbox{where}\quad
\mathcal{L'}u=\left[ \partial_{i} \left(\overline{A}_{ij}^{\alpha \beta}(x) \partial_{j} u_{\beta}\right)
+\overline{B}_{i}^{\alpha \beta}(x) \partial_{i}u_{\beta}\right]_{\alpha},
\end{equation}
with coefficients $\overline{A}$ as in \eqref{eq-Abar} and coefficients $\overline{B}$ satisfying a simple estimate $|B|\lesssim |\nabla\lambda|+|\nabla\mu|$. The operator $\mathcal L'$ is strongly elliptic
if 
$$\mbox{\rm ess }\inf_{x\in\Omega}\{(\sqrt{8}-1)\mu(x)+\lambda(x),(\sqrt{8}+1)\mu(x)-\lambda(x)\}>0.$$
\end{lemma}
Observe that this condition also implies that $\mbox{\rm ess }\inf_{x\in\Omega} \mu>0$ and hence we do not have to state that explicitely. By combining Theorem \ref{S3:T1} and Lemma \ref{Lamma-ell} we see that Corollary \ref{MM:C1} follows.
\vglue3mm

\end{document}